%% file: chaumontfrelet_2022b.tex
\documentclass[a4paper,11pt]{amsart}

\usepackage[breaklinks,bookmarks=false]{hyperref}
\hypersetup{colorlinks, linkcolor=blue, citecolor=blue,urlcolor=blue,plainpages=false}

\usepackage{a4wide}
\usepackage{mathabx}
\usepackage[foot]{amsaddr}

\input{general_commands}

\input{special_commands}

\newtheorem{theorem}{Theorem}
\newtheorem{lemma}[theorem]{Lemma}
\newtheorem{corollary}[theorem]{Corollary}
\newtheorem{remark}[theorem]{Remark}

\numberwithin{theorem}{section}
\numberwithin{equation}{section}

\title[IPDG under minimal regularity and application to Helmholtz problems]%
{Duality analysis of interior penalty discontinuous Galerkin methods under minimal regularity
and application to the a priori and a posteriori error analysis of Helmholtz problems}

\author{T. Chaumont-Frelet$^\star$}

\address{\vspace{-.5cm}}
\address{\noindent \tiny \textup{$^\star$Inria, Univ. Lille, CNRS, UMR 8524 -- Laboratoire Paul Painlev\'e}}

\newcommand{\revision}[1]{#1}
\newcommand{\blunder}[1]{#1}

\begin{document}

\begin{abstract}
We consider interior penalty discontinuous Galerkin discretizations of time-harmonic
wave propagation problems modeled by the Helmholtz equation, and derive novel a priori
and a posteriori estimates. Our analysis classically relies on duality arguments of Aubin--Nitsche
type, and its originality is that it applies under minimal regularity assumptions. The
estimates we obtain directly generalize known results for conforming discretizations,
namely that the discrete solution is optimal in a suitable energy norm and that the error
can be explicitly controlled by a posteriori estimators, provided the mesh is sufficiently fine.

\vspace{.5cm}
\end{abstract}

\maketitle

\section{Introduction}

Discontinuous Galerkin methods are a popular approach to discretize PDE
boundary value problems \cite{dipietro_ern_2012a}. Similar to conforming finite element methods
\cite{ciarlet_2002a,ern_guermond_2019a}, they can handle general meshes, which allows to account for complicated
geometries. In addition, the use of discontinuous polynomial shape functions
facilitates the presence of hanging nodes and varying polynomial degree.
As a result, discontinuous Galerkin methods are especially suited for $hp$-adaptivity
\cite{congreve_gedicke_perugia_2019a,houston_schotzau_wihler_2006a}.
Furthermore, since all their degrees of freedom are attached with mesh cells,
discontinuous Galerkin methods allow linear memory access, which is crucial for
efficient computer implementations, in particular on GPUs
\cite{chan_wang_modave_remacle_warburton_2016a}.

In the context of wave propagation problems, discontinuous Galerkin methods
further display specific advantages. For time-dependent problems, the resulting
mass-matrix is block-diagonal \cite{agut_diaz_2013a,grote_schneebeli_schotzau_2006a},
which enables explicit time-stepping schemes without mass-lumping
\cite{cohen_joly_roberts_tordjman_2001a}. For time-harmonic wave propagation too,
discontinuous Galerkin formulations are attractive as they exhibit additional stability
as compared to conforming alternatives
\cite{bernkopf_sauter_torres_veit_2022a,feng_wu_2009a,feng_wu_2011a}.
These interesting properties have henceforth motivated a large number of works
considering discontinuous Galerkin discretizations of wave propagation problems,
and a non-exhaustive list includes \cite{agut_diaz_2013a,congreve_gedicke_perugia_2019a,%
du_zhu_2016a,feng_wu_2009a,feng_wu_2011a,grote_schneebeli_schotzau_2006a,hoppe_sharma_2013a,%
sauter_zech_2015a}.

In this work, we consider the acoustic Helmholtz equation, which is probably the
simplest model problem relevant to the difficulties of wave propagation.
Specifically, given a domain $\Omega \subset \R^d$, $d=2$ or $3$, and $f: \Omega \to \C$,
the unknown $u: \Omega \to \C$ should satisfy
\begin{equation}
\label{eq_helmholtz_strong}
\left \{
\begin{array}{rcll}
-\omega^2 \mu u-\div \left (\BA\grad u\right ) &=& f & \text{ in } \Omega,
\\
u &=& 0 & \text{ on } \GD,
\\
\BA\grad u \cdot \bn &=& 0 & \text{ on } \GN,
\\
\BA\grad u \cdot \bn - i\omega\gamma u &=& 0 & \text{ on } \GR,
\end{array}
\right .
\end{equation}
where $\overline{\GD} \cup \overline{\GN} \cup \overline{\GR} = \partial \Omega$
is a partition of the boundary and $\mu,\BA$ and $\gamma$ are given coefficients.
Precise assumptions are listed in Section \ref{section_setting}.

Our interest lies in interior penalty discontinuous Galerkin (IPDG) discretizations
of \eqref{eq_helmholtz_strong}. In particular, we focus on the ``minimal regularity''
case, where we do not assume any specific smoothness for the solution $u$. To the best
of our knowledge, this problem has not been considered in the literature, and available
works essentially assumes that the solution belongs to $H^{3/2+\varepsilon}$, so that the
traces of $\grad u$ are well-defined on mesh faces, see e.g. \cite[Eq. (4.5)]{hoppe_sharma_2013a}
and \cite[Lemma 2.6]{sauter_zech_2015a}. Unfortunately, such assumptions rule out important
configurations of coefficients and boundary conditions, which may bring the regularity of the
solution arbitrarily close to $H^1$ (see the appendix of \cite{costabel_dauge_nicaise_1999a}
for instance).

When considering conforming finite element discretizations of \eqref{eq_helmholtz_strong},
the so-called ``Schatz argument'' enables to show that the discrete solution $u_h$ is
quasi-optimal if the mesh is fine enough \cite{schatz_1974a}. Assuming for simplicity that
$\GR = \emptyset$ and introducing the energy norm
\begin{equation*}
\enorm{v}_{\omega,\Omega}^2
\eq
\omega^2 \|v\|_{\mu,\Omega}^2 + \|\grad v\|_{\BA,\Omega}^2
\quad
v \in H^1_{\GD}(\Omega),
\end{equation*}
we have
\begin{equation}
\label{eq_intro_apriori}
\enorm{u-u_h}_{\omega,\Omega}
\leq
\frac{1}{1-\gba^2} \min_{v_h \in V_h} \enorm{u-v_h}_{\omega,\Omega},
\end{equation}
whenever the approximation factor
\begin{equation}
\label{eq_intro_gba}
\gba \eq 2\omega \max_{\substack{\psi \in L^2(\Omega) \\ \|\psi\|_{\mu,\Omega} = 1}}
\min_{v_h \in V_h} \enorm{u_\psi-v}_{\omega,\Omega}
\end{equation}
is strictly less than one (all these notations are detailed in Section \ref{section_preliminary}
below). Above, $u_\psi$ solves \eqref{eq_helmholtz_strong}
with right-hand side $\mu\psi$ instead of $f$. Similarly, when considering a posteriori
error estimation \cite{chaumontfrelet_ern_vohralik_2021a}, we have
\begin{equation}
\label{eq_intro_aposteriori}
\enorm{u-u_h}_{\omega,\Omega} \leq \sqrt{1+\gba^2} \eta,
\end{equation}
where $\eta$ is a suitable a posteriori estimator.

Estimates similar to \eqref{eq_intro_apriori} and \eqref{eq_intro_aposteriori}
are available for IPDG discretizations \cite{melenk_parsania_sauter_2013a,sauter_zech_2015a},
but with energy norms involving the normal trace of the gradient on faces, thus essentially
requiring $H^{3/2+\varepsilon}$ regularity of $u_\psi$. Here, in contrast, we extend
\eqref{eq_intro_apriori} and \eqref{eq_intro_aposteriori} to IPDG discretizations without
additional regularity assumptions on the solutions $u_\psi$. As detailed below, our key
finding is that it can be achieved by redefining the approximation factor as
\begin{equation}
\label{eq_intro_gba_IPDG}
\gba^2
\eq
4\omega^2
\max_{\substack{\psi \in L^2(\Omega) \\ \|\psi\|_{\mu,\Omega} = 1}}
\left (
\min_{v_h \in V_h} \enorm{u_\psi-v_h}_{\revision{\dagger,1},\CT_h}^2
+
\min_{\bw_h \in \BW_h} \enorm{\BA\grad u_\psi-\bw_h}_{\revision{\dagger,\ddiv},\CT_h}^2
\right ),
\end{equation}
where $\BW_h$ is the BDM finite element space built using the same mesh and polynomial degree
than $V_h$, and $\enorm{{\cdot}}_{\revision{\dagger,1},\CT_h}$ and $\enorm{{\cdot}}_{\revision{\dagger,\ddiv},\CT_h}$
are $H^1_{\GD}(\Omega)$ and $\BH_{\GN}(\ddiv,\Omega)$ norms appropriately
scaled by the mesh size. The additional term in \eqref{eq_intro_gba_IPDG}
as compared to \eqref{eq_intro_gba} is necessary to account for the non-conformity
of the scheme.

Actually, the interest of the subject exceeds time-harmonic wave propagation.
In fact, the a priori and a posteriori error analysis of finite element discretizations
to \eqref{eq_helmholtz_strong} rely on duality arguments of Aubin-Nitsche type
\cite{schatz_1974a}. Such techniques are crucial in time-harmonic wave propagation
\cite{melenk_sauter_2010a,sauter_zech_2015a}, but there also useful in other contexts
to establish convergence in weak norms, see e.g.
\cite[Section 5.1]{arnold_brezzi_cockburn_marini_2002a}.
To the best of our knowledge, duality analysis for IPDG discretizations under
minimal regularity has not been addressed in the literature, and it is our goal
to do so here.

%%%   We finally mention related recent works dealing with similar issues
%%%   \cite{}.

The remainder of this work is organized as follows. In Section \ref{section_preliminary},
we precise the setting and introduce key notations. Section \ref{section_duality}
presents the key argument that enables duality techniques for IPDG under minimal regularity.
Sections \ref{section_apriori} and \ref{section_aposteriori} then employ the aforementioned
reasoning to perform the a priori and a posteriori error analysis of IPDG discretization for
Helmholtz problems.
\revision{Finally, for the sake of completeness, we provide an estimate
for the approximation factor $\gba$ in Appendix \ref{appendix_approximation_factors}.}

\section{Setting and preliminary results}
\label{section_preliminary}

\subsection{Setting}
\label{section_setting}

Throughout this work, $\Omega \subset \R^d$, with $d=2$ or $3$, is a Lipschitz
polytopal domain. The boundary of $\Omega$ is partitioned into three open, Lipschitz
and disjoint polytopal subsets $\GD$, $\GN$ and $\GR$ such that
$\partial \Omega = \overline{\GD} \cup \overline{\GN} \cup \overline{\GR}$.
We employ the notation $\bn$ for the unit vector normal to $\partial \Omega$
pointing outside $\Omega$.

We consider coefficients $\mu: \Omega \to \R$, $\BA: \Omega \to \R^{d \times d}$
and $\gamma: \GR \to \R$ satisfying the following properties. We assume that there
exists a partition $\LP$ of $\Omega$ into a finite number disjoint open polytopal
subdomains such that $\mu|_P = \mu_P \in \R$ and $\BA|_P = \BA_P \in \R^{d \times d}$ take constant
values for all $P \in \LP$. Similarly, there exists a finite partition $\LQ$ of $\GR$
consisting of open polytopal subsets such that $\gamma|_Q = \gamma_Q \in \R$ is constant
for each $Q \in \LQ$. For each $P \in \LP$, we assume that $\BA_P$ is symmetric, and let
$\alpha_P \eq \min_{\bxi \in \R^d; |\bxi| = 1} \BA_P \bxi \cdot \bxi$. We then classically
require that $\min_{P \in \LP} \mu_P > 0$, $\min_{P \in \LP} \alpha_P > 0$, and
$\min_{Q \in \LQ} \gamma_Q > 0$.

We also fix a real number $\omega > 0$ representing the (angular) frequency.

\subsection{Functional spaces}

If $D \subset \mathbb R^d$ is an open set, we denote by $L^2(D)$ the Lebesgue space
of complex-valued square integrable functions defined over $D$, and we set
$\BL^2(D) \eq [L^2(D)]^d$ for vector-valued functions. The notations $(\cdot,\cdot)_D$
and $\|{\cdot}\|_D$ then stand for the usual inner-product and norm of $L^2(D)$ or $\BL^2(D)$.
In addition, if $w: D \to \mathbb R$ is a measurable function satisfying
$0 < \operatorname{ess} \inf_D w$ and $\operatorname{ess} \sup_D w < +\infty$,
then $\|{\cdot}\|_{w,D}^2 \eq (w \cdot,\cdot)_D$ defines a norm on $L^2(D)$ equivalent
to the standard one. We use the same notation in $\BL^2(D)$ with matrix-valued weights.
Besides, we employ similar notations for $d-1$ manifolds.

$H^1(D)$ is the Sobolev space of functions $v \in L^2(D)$ such that $\grad v \in \BL^2(D)$,
where $\grad$ denotes the weak gradient defined in the sense of distributions. If
$\Gamma \subset \partial D$ is a relatively open subset of the boundary of $D$, then
$H^1_\Gamma(D)$ stands for the subset of functions $v \in H^1(D)$ such that
$v|_{\Gamma} = 0$ in the sense of traces. We refer the reader to \cite{adams_fournier_2003a}
for a detailed description of the above spaces. On $H^1_{\GD}(\Omega)$, we will often employ
the following ``energy'' norm:
\begin{equation}
\label{eq_energy_norm}
\enorm{v}_{\omega,\Omega}^2
\eq
\omega^2\|v\|_{\mu,\Omega}^2 + \omega\|v\|_{\gamma,\GR}^2 + \|\grad v\|_{\BA,\Omega}^2
\quad
\forall v \in H^1_{\GD}(\Omega).
\end{equation}

It turns out that Sobolev spaces of vector-valued functions will be key in the duality
analysis we are about to perform. Specifically, we will need the space $\BH(\ddiv,D)$
of functions $\bv \in \BL^2(D)$ such that $\div \bv \in L^2(D)$ where $\div$ is the weak
divergence operator \cite{girault_raviart_1986a}. Following, e.g., \cite{fernandes_gilardi_1997a},
if $\Gamma \subset \partial D$ is a relatively open set, the normal trace $(\bw \cdot \bn)|_\Gamma$
of $\bw \in \BH(\ddiv,D)$ can be defined in a weak sense, and $\BH_\Gamma(\ddiv,D)$
will stand for the space of $\bw \in \BH(\ddiv,D)$ such that $(\bw \cdot \bn)|_\Gamma = 0$.

We are studying ``Robin-type'' boundary conditions which are not naturally handled in
the $\BH(\ddiv)$ setting. We follow the standard remedy \cite{chaumontfrelet_2019a,monk_2003a}
and introduce the space
\begin{equation*}
\BX(\ddiv,\Omega)
\eq
\left \{
\bw \in \BH(\ddiv,\Omega) \; | \; (\bw \cdot \bn)|_{\GR} \in L^2(\GR)
\right \},
\end{equation*}
where additional normal trace regularity is enforced on the Robin boundary.
The notation $\BX_{\GN}(\ddiv,\Omega) \eq \BX(\ddiv,\Omega) \cap \BH_{\GN}(\ddiv,\Omega)$
will also be useful.

\subsection{Helmholtz problem}

Central to our considerations will be the sesquilinear form
\begin{equation}
\label{eq_b}
b(\phi,v)
\eq
-\omega^2 (\mu\phi,v)_\Omega
-i\omega(\gamma\phi,v)_{\GR}
+(\BA\grad \phi,\grad v)_\Omega
\qquad
\forall \phi,v \in H^1_{\GD}(\Omega),
\end{equation}
corresponding to the weak formulation of \eqref{eq_helmholtz_strong}.
We will assume throughout this work that the considered Helmholtz
problem is well-posed, meaning that there exists $\gst > 0$ such that
\begin{equation}
\label{eq_well_posed}
\min_{\substack{\phi \in H^1_{\GD}(\Omega) \\ \enorm{\phi}_{\omega,\Omega} = 1}}
\max_{\substack{ v   \in H^1_{\GD}(\Omega) \\ \enorm{ v  }_{\omega,\Omega} = 1}}
\Re b(\phi,v)
=
\frac{1}{\gst}.
\end{equation}
In our setting, \eqref{eq_well_posed} always holds as soon as $\GR$ has a positive
measure due to the unique continuation principle. On the other hand, if $\GR = \emptyset$,
then \eqref{eq_well_posed} fails to hold if and only if $\omega^2$ is an eigenvalue of the
resulting self-adjoint operator. In general, it is hard to quantitatively estimate $\gst$, but a
reasonable assumption is that it grows polynomially with the frequency
\cite{lafontaine_spence_wunsch_2021a}. It is also worth mentioning that the
constant may be explicitly controlled in some specific configurations
\cite{%
barucq_chaumontfrelet_gout_2017a,%
chandlerwilde_spence_gibbs_smyshlyaev_2020a,%
chaumontfrelet_spence_2021a,%
moiola_spence_2019a}.

In the remainder of this work, we fix a right-hand side $f \in L^2(\Omega)$,
and let $u \in H^1_{\GD}(\Omega)$ be the unique element satisfying
\begin{equation}
\label{eq_helmholtz_weak}
b(u,v) = (f,v)_\Omega \quad \forall v \in H^1_{\GD}(\Omega).
\end{equation}

%%% TODO: remark general rhs?

\subsection{Computational mesh}

We consider a mesh $\CT_h$ of $\Omega$ consisting of non-overlapping (closed) simplicial
elements $K$. We classically assume that $\CT_h$ is a matching mesh meaning that
the intersection of two distinct elements either is empty, or it is a full face,
edge or vertex of the two elements (see, e.g. \cite[Section 2.2]{ciarlet_2002a} or
\cite[Definition 6.11]{ern_guermond_2019a}).
The set of faces of the mesh is denoted by $\CF_h$. We also employ the standard notations
$h_K$ and $\rho_K$ for the diameter of $K \in \CT_h$ and the diameter of the largest ball
contained in $K$ (see \cite[Theorem 3.1.3]{ciarlet_2002a} or
\cite[Definition 6.4]{ern_guermond_2019a}).
Then, $\kappa_K \eq h_K/\rho_K$ is the shape-regularity parameter of $K \in \CT_h$, and
$\kappa \eq \max_{K \in \CT_h} \kappa_K$. We also introduce the global mesh size
$h \eq \max_{K \in \CT_h} h_K$. Similarly, $h_F$ stands for the diameter of the face
$F \in \CF_h$.

We further assume that the mesh $\CT_h$ conforms with the partition of the boundary
and coefficients in the sense that, for each $K \in \CT_h$, there exists a (unique)
$P \in \LP$ such that $K \subset \overline{P}$ and for each $F \in \CF_h$ with
$F \subset \partial \Omega$, we have either $F \subset \overline{\GD}$,
$F \subset \overline{\GN}$ or $F \subset \overline{\GR}$.
If $F \subset \overline{\GR}$, we additionally require that $F \subset \overline{Q}$
for some (unique) $Q \in \LQ$. We respectively denote by $\CF_h^{\rm D}$,
$\CF_h^{\rm N}$ and $\CF_h^{\rm R}$ the set of faces $F \in \CF_h$ such that
$F \subset \overline{\GD}$, $\overline{\GN}$ or $\overline{\GR}$. We also set
$\CF_h^{\rm e} \eq \CF_h^{\rm D} \cup \CF_h^{\rm N} \cup \CF_h^{\rm R}$,
and $\CF^{\rm i} \eq \CF_h \setminus \CF_h^{\rm e}$.
For $K \in \CT_h$, we can then set $\mu_K \eq \mu_P$, $\alpha_K \eq \alpha_P$, and
$\vel_K \eq \sqrt{\alpha_K/\mu_K}$, where $P \in \LP$ contains $K$.  Similarly, if
$F \in \CF_h^{\rm e}$, we set $\alpha_F \eq \alpha_K$ where $K \in \CT_h$ is the
only element having $F$ as a face. If $F \in \CF_h^{\rm R}$, we also introduce
$\gamma_F \eq \gamma_Q$ and $\vel_F \eq \alpha_F/\gamma_F$ where $Q$ is the unique
subset in $\LQ$ containing $F$.

Finally, we associate with each $F \in \CF_h$ a unit normal vector $\bn_F$. If
$F \in \CF_h^{\rm e}$, we require that $\bn_F = \bn$. Otherwise, the orientation
of $\bn_F$ is arbitrary, but fixed.

\revision
{
\begin{remark}[Hanging nodes]
For the sake of simplicity, our assumptions on the mesh rule out the presence
of hanging nodes. However, this is not essential for the analysis performed in
the manuscript. Instead, the key assumption required is that the conforming Lagrange
and Raviart-Thomas finite element spaces associated with the mesh are sufficiently rich,
which is known to be the case in a variety of situations, see e.g.
\cite{ainsworth_pinchedez_2002a,bonito_nochetto_2010a}. Since the precise assumptions
are intricate, we have chosen to restrict our attention to matching meshes to ease the
presentation. The interested reader will find more details in Appendix
\ref{appendix_approximation_factors} below. Related to this comment,
notice that the present analysis would be harder to extend to polytopal meshes.
\end{remark}
}

\subsection{Polynomial spaces}

In the remainder of this work, we fix a polynomial degree $p \geq 1$.
For $K \in \CT_h$, $\CP_p(K)$ is the set of polynomial functions $K \to \C$
of total degree less than or equal to $q$. For vector-valued functions, we also set
$\BCP_p(K) \eq [\CP_p(K)]^3$. If $\CT \subset \CT_h$ is a collection of elements
covering the (open) domain $\Theta$, then
$\CP_p(K) \eq \{ v \in L^2(\Theta); \; v|_K \in \CP_p(K) \; \forall K \in \CT\}$
and $\BCP_p(\CT) \eq [\CP_p(\CT)]^3$.

\subsection{Broken Sobolev spaces}

% The following broken spaces will be useful.
We define $H^1(\CT_h)$ as the
subset of functions $v \in L^2(\Omega)$ such that $v|_K \in H^1(K)$ for all $K \in \CT_h$.
For functions in $H^1(\CT_h)$, we still employ the notations $\grad$ for the element-wise
weak gradient, so that $\grad (H^1(\CT_h)) \subset \BL^2(\Omega)$. To avoid confusions, we
will employ the alternative notations $(\cdot,\cdot)_{\CT_h}$ and $\|{\cdot}\|_{\CT_h}$ for
the $L^2(\Omega)$ and $\BL^2(\Omega)$ norms and inner-products when working with broken functions
(we also employ the same notation for weighted norms). Similarly, we write
\begin{equation*}
\|v\|_{\gamma,\CF_h^{\rm R}}^2 \eq \sum_{F \in \CF_h^{\rm R}} \|v\|_{\gamma,F}^2
\quad
\forall v \in L^2(\GR).
\end{equation*}
The notation
\begin{equation*}
(\phi,v)_{\CF_h} \eq \sum_{F \in \CF_h} (\phi,v)_F
\qquad
\forall \phi,v \in \bigoplus_{F \in \CF_h} L^2(F)
\end{equation*}
will also be useful.

\subsection{Jumps, averages, lifting and discrete gradient}

Considering $\phi \in H^1(\CT_h)$, its jump through a face $F \in \CF_h^{\rm i}$
is defined by
\begin{equation*}
\jmp{\phi}_F \eq \phi_+|_F \bn_+ \cdot \bn_F + \phi_-|_F \bn_- \cdot \bn_F
\end{equation*}
with $K_\pm \in \CT_h$ the two elements such that $F = \partial K_+ \cap \partial K_-$,
$\phi|_\pm \eq \phi|_{K_\pm}$ and $\bn_\pm \eq \bn_{K_\pm}$.
For an exterior face $F \in \CF_h^{\rm e}$, we set instead
\begin{equation*}
\jmp{\phi}_F \eq \phi|_F \text{ if } F \in \CF_h^{\rm D}
\quad
\text{ and }
\quad
\jmp{\phi}_F \eq 0 \text{ otherwise.}
\end{equation*}
Similarly, if $\bw \in \BCP_p(\CT_h)$, its average on $F \in \CF_h$ is given by
\begin{equation*}
\avg{\bw}_F \eq \frac{1}{2} (\bw_+|_F + \bw_-|_F) \quad \text{ if } F \in \CF_h^{\rm i}
\quad \text{ and } \quad
\avg{\bw}_F \eq \bw|_F \quad \text{ if } F \in \CF_h^{\rm e},
\end{equation*}
where $\bw_\pm \eq \bw|_{K_\pm}$ for the two elements $K_\pm \in \CT_h$ such that
$F = K_- \cap K_+$ in the case of an interior face.

A key part of our analysis will be to give a meaning to the terms
$(\jmp{\phi},\avg{\BA\grad v} \cdot \bn_F)_{\CF_h}$ appearing in the
IPDG form, for functions $\phi$ and $v$ only belonging to $H^1(\CT_h)$.
This is subtle, since the normal trace of $\grad v$ is actually not
defined on faces. Following \cite[Section 4.3]{dipietro_ern_2012a},
the solution is to introduce a lifting operator defined as follows.

For $\phi \in H^1(\CT_h)$ we define the lifting $\LIFT{\phi}$ as the
unique element of $\BCP_p(\CT_h)$ such that
\begin{equation*}
(\LIFT{\phi},\bw_h)_{\CT_h}
=
(\jmp{\phi}_F,\avg{\bw_h} \cdot \bn_F)_{\CF_h}
\qquad
\forall \bw_h \in \BCP_p(\CT_h).
\end{equation*}
Notice that we can then write $(\BA\LIFT{\phi},\grad v)_{\CT_h}$
instead of $(\jmp{\phi},\avg{\BA\grad v}\cdot\bn_F)_{\CF_h}$
for functions $\phi, v \in \CP_p(\CT_h)$, with the advantage that
the second expression is still well-defined for general $H^1(\CT_h)$ arguments.

The notion of weak gradient will also be useful. Specifically, we set
\begin{equation*}
\GRAD(\phi) \eq \grad \phi - \LIFT{\phi} \in \BL^2(\CT_h)
\end{equation*}
for all $\phi \in H^1(\CT_h)$. Importantly, we have $\GRAD(\phi) = \grad \phi$
whenever $\phi \in H^1_{\GD}(\Omega)$, and
\begin{equation*}
(\GRAD(\phi),\bw_h)_{\CT_h} + (\phi,\div \bw_h)_{\CT_h} = (\phi,\bw_h \cdot \bn)_{\GR}
\end{equation*}
for all $\phi \in H^1(\CT_h)$ and $\bw_h \in \BCP_p(\CT_h) \cap \BX_{\GN}(\revision{\ddiv,\Omega})$.

\subsection{Broken norms}

For $v \in H^1(\CT_h)$, the broken counterpart of the energy norm introduced in
\eqref{eq_energy_norm} is defined by
\begin{equation*}
\enorm{v}^2_{\omega,\CT_h}
\eq
\omega^2 \|v\|_{\mu,\CT_h}^2
+
\omega \|v\|_{\gamma,\CF_h^{\rm R}}^2
+
\|\GRAD(v)\|_{\BA,\CT_h}^2.
\end{equation*}
Notice that it is equal to the $\enorm{{\cdot}}_{\omega,\Omega}$ norm if
$v \in H^1_{\GD}(\Omega)$. We will also employ the mesh-dependent norms
\begin{equation*}
\|v\|_{\dagger,1,\CT_h}^2
\eq
\sum_{K \in \CT_h} \left \{
\max \left (
1,\frac{\omega^2 h_K^2}{\vel_K^2}
\right )
\frac{\alpha_K}{h_K^2}\|v\|_K^2
+
\|\GRAD(v)\|_{\BA,K}^2
\right \}
+
\sum_{F \in \CF_h^{\rm R}}
\max \left (1,\frac{\omega h_F}{\vel_F}\right )
\frac{\alpha_F}{h_F} \|v\|_F^2
\end{equation*}
and
\begin{equation*}
\|\bw\|_{\dagger,\ddiv,\CT_h}^2
\eq
\sum_{K \in \CT_h} \left \{
\|\bw\|_{\BA^{-1},K}^2
+
\frac{h_K^2}{\alpha_K} \|\div \bw\|_{K}^2
\right \}
+
\sum_{F \in \CF_h^{\rm R}} \frac{h_F}{\alpha_F} \|\bw \cdot \bn\|_F^2
\end{equation*}
for $\bw \in \BX_{\GN}(\ddiv,\Omega)$. These norms are ``dual'' to each other in the sense that
\begin{equation}
\label{eq_norm_dual}
|(\GRAD(\phi),\bw)_{\CT_h} + (\phi,\div \bw)_{\CT_h} - (\phi,\bw\revision{\cdot \bn})_{\CF_h^{\rm R}}|
\leq
\|\phi\|_{\dagger,1,\CT_h}\|\bw\|_{\dagger,\ddiv,\CT_h}
\end{equation}
for all $\phi \in H^1(\CT_h)$ and $\bw \in \BX_{\GN}(\ddiv,\Omega)$.
For future references, we notice that the ``$\max$'' in the definition of the
$\|{\cdot}\|_{\dagger,1,\CT_h}$ are chosen so that
\begin{equation}
\label{eq_norm_control}
\enorm{v}_{\omega,\CT_h} \leq \|v\|_{\dagger,1,\CT_h} \qquad \forall v \in H^1(\CT_h).
\end{equation}

\subsection{IPDG form}

Following \cite[Section 4.3.3]{dipietro_ern_2012a}, our DG approximation of the
$(\BA\grad\cdot,\grad\cdot)_\Omega$ form is given by
\begin{equation}
\label{eq_IPDG_discrete_gradient}
a_h(\phi,v)
\eq
(\BA\GRAD(\phi),\GRAD(v))_{\CT_h} + s_h(\phi,v)
\qquad
\forall \phi,v \in H^1(\CT_h)
\end{equation}
where $s_h: H^1(\CT_h) \times H^1(\CT_h) \to \C$ is a sesqulinear form
satisfying $s_h(\phi,v) = 0$ whenever $\phi \in H^1_{\revision{\GD}}(\Omega)$ or
$v \in H^1_{\revision{\GD}}(\Omega)$. Importantly, we have
\begin{equation}
\label{eq_consistency}
a_h(\phi,v) = (\BA\grad\phi,\grad v)_\Omega
\end{equation}
when $\phi,v \in H^1_{\GD}(\Omega)$.

It is worthy to note that the IPDG form is usually not presented (and not implemented)
as it is presented in \eqref{eq_IPDG_discrete_gradient}. Instead
\cite{agut_diaz_2013a,feng_wu_2009a,grote_schneebeli_schotzau_2006a,sauter_zech_2015a},
the method is usually written as
\begin{equation}
\label{eq_IPDG_jump}
a_h(\phi,v)
=
(\BA\grad\phi,\grad v)_{\CT_h}
-
(\avg{\BA\grad \phi}\cdot \bn_F,\jmp{v})_{\CF_h}
-
(\jmp{\phi},\avg{\BA\grad v}\cdot \bn_F)_{\CF_h}
+
\sum_{F \in \CF_h^{\rm i} \cup \CF_h^{\rm D}} \frac{\beta_F}{h_F}(\jmp{\phi},\jmp{v})_F
\end{equation}
where $\beta_F \sim p^2$ is a penalty parameter chosen to be large enough
\cite{agut_diaz_2013a}.
However, as shown in \cite[Section 4.3.3]{dipietro_ern_2012a},
the formulation of \eqref{eq_IPDG_jump} can be recast in the framework of
\eqref{eq_IPDG_discrete_gradient} by setting
\begin{equation}
\label{eq_penalization_usual}
s_h(\phi,v)
\eq
\sum_{F \in \CF_h^{\rm i} \cup \CF_h^{\rm D}} \frac{\beta_F}{h_F}(\jmp{\phi},\jmp{v})_F
-(\BA\LIFT{\phi},\LIFT{v})_\Omega.
\end{equation}

\subsection{The discrete Helmholtz problem}

Consider the sesquilinear form 
\begin{equation*}
b_h(\phi,v)
\eq
-\omega^2 (\mu\phi,v)_{\CT_h}
-i\omega (\gamma\phi,v)_{\CF_h^{\rm R}}
+a_h(\phi,v)
\qquad
\forall \phi,v \in H^1(\CT_h).
\end{equation*}
Then, the discrete problem consists in finding $u_h \in \CP_p(\CT_h)$ such that
\begin{equation}
\label{eq_helmholtz_ipdg}
b_h(u_h,v_h) = (f,v_h)_{\CT_h} \qquad \forall v_h \in \CP_p(\CT_h).
\end{equation}
For any $u_h$ satisfying \eqref{eq_helmholtz_ipdg}, recalling \eqref{eq_helmholtz_weak}
and due to \eqref{eq_consistency}, the Galerkin orthogonality property
\begin{equation}
\label{eq_galerkin_orthogonality}
b_h(u-u_h,w_h) = 0
\end{equation}
holds true for all discrete conforming test functions $w_h \in \CP_p(\CT_h) \cap H^1_{\GD}(\Omega)$.
Similarly, we have
\begin{equation}
\label{eq_continuity}
|b_h(\phi,v)|
\leq
\enorm{\phi}_{\omega,\CT_h}\enorm{v}_{\omega,\CT_h}
\end{equation}
for $\phi,v \in H^1(\CT_h)$ whenever $\phi$ or $v$ belongs to $H^1_{\GD}(\Omega)$.

%%% TODO: discuss the well-posedness

\subsection{Conforming subspaces and projections}

The sets $\CP_p(\CT_h) \cap H^1_{\GD}(\Omega)$ and $\BCP_p(\CT_h) \cap \BX_{\GN}(\ddiv,\Omega)$
are the usual Lagrange and BDM finite element spaces (see, e.g.,
\cite[Sections 8.5.1 and 11.5]{ern_guermond_2019a}). Although they are not needed to implement
the IPDG discretization, they will be extremely useful for the analysis.
The projections
\begin{equation*}
\pihg v
\eq
\arg \min_{v_h \in \CP_p(\CT_h) \cap H^1_{\GD}(\Omega)}
\|v-v_h\|_{\revision{\dagger,1},\CT_h}
\end{equation*}
and
\begin{equation*}
\pihd \bw
\eq
\arg \min_{\bw_h \in \BCP_p(\CT_h) \cap \BX_{\GN}(\ddiv,\Omega)}
\|\bw-\bw_h\|_{\revision{\dagger,\ddiv},\CT_h}
\end{equation*}
are well-defined for all $v \in H^1_{\GD}(\Omega)$ and $\bw \in \BX_{\GN}(\ddiv,\Omega)$,
since the mesh-dependent norms appearing in the right-hand sides are naturally associated
with inner-product. We will also need the conforming projector
\begin{equation*}
\pig v \eq \arg \min_{s \in H^1_{\GD}(\Omega)} \|v-s\|_{\revision{\dagger,1},\CT_h}
\end{equation*}
defined for all $v \in H^1(\CT_h)$.

\subsection{Approximation factors}

\begin{subequations}
\label{eq_definition_gba}
Following
\cite{chaumontfrelet_ern_vohralik_2021a,
melenk_parsania_sauter_2013a,
melenk_sauter_2010a,
sauter_zech_2015a}, our analysis will rely on duality arguments and
approximation factors will be a central concept.
For $\psi \in L^2(\Omega)$ and $\Psi \in L^2(\GR)$,
let $u^\star_\psi,U^\star_\Psi \in H^1_{\revision{\GD}}(\Omega)$ solve
\begin{equation*}
b(w,u^\star_\psi) = \omega(\mu w,\psi)_\Omega,
\qquad
b(w,U^\star_\Psi) = \omega^{1/2}(\gamma w,\Psi)_{\GR}
\qquad
\forall w \in H^1_{\revision{\GD}}(\Omega).
\end{equation*}
The approximation factors
\begin{equation}
\label{eq_definition_gbag}
\cgbag
\eq
\max_{\substack{\psi \in L^2(\Omega) \\ \|\psi\|_{\mu,\Omega} = 1}}
\|u_\psi^\star-\pihg u_\psi^\star\|_{\dagger,1,\CT_h},
\quad
\tgbag
\eq
\max_{\substack{\Psi \in L^2(\GR) \\ \|\Psi\|_{\gamma,\GR} = 1}}
\|U_\Psi^\star-\pihg U_\Psi^\star\|_{\dagger,1,\CT_h},
\end{equation}
where previously employed in \cite{chaumontfrelet_ern_vohralik_2021a}
for the a posteriori error analysis of conforming discretizations.
Here, we will additionally need divergence-conforming approximation factors
\begin{equation}
\label{eq_definition_gbad}
\begin{aligned}
\cgbad &\eq \max_{\substack{\psi \in L^2(\Omega) \\ \|\psi\|_{\mu,\Omega} = 1}}
\|\BA\grad u_\psi^\star-\pihd (\BA\grad u_\psi^\star)\|_{\dagger,\ddiv,\CT_h},
\\
\tgbad &\eq \max_{\substack{\Psi \in L^2(\GR) \\ \|\Psi\|_{\gamma,\GR} = 1}}
\|\BA\grad U_\Psi^\star-\pihd (\BA\grad U_\Psi^\star)\|_{\dagger,\ddiv,\CT_h},
\end{aligned}
\end{equation}
to deal with the non-conformity of the IPDG approximation under minimal regularity.
We finally introduce the following short-hand notations
\begin{align}
\label{eq_definition_gba_summed}
\gbag^2 &\eq 4\cgbag^2 + 2\tgbag^2,
\quad
&\gbad^2 &\eq 4\cgbad^2 + 2\tgbad^2,
\\
\cgba^2 &\eq \cgbag^2 + \cgbad^2,
\quad
&\tgba^2 &\eq \tgbag^2 + \tgbad^2,
\end{align}
and
\begin{equation}
\gba^2 \eq \gbag^2 + \gbad^2 = 4\cgba^2 + 2\tgba^2.
\end{equation}
\end{subequations}

By combining elliptic regularity and approximation properties of finite element
spaces, it is easy to see that $\gba \to 0$ as $h/p \to 0$. However, the dependency
on the PDE coefficients and on the frequency may be hard to track. In particular,
the stability constant $\gst$ typically plays a central role in estimating $\gba$.
Qualitative estimates for $\gba$ can be found in
\cite{chaumontfrelet_nicaise_2020a,melenk_sauter_2010a}, and explicit estimates
are available in specific situations \cite{chaumontfrelet_ern_vohralik_2021a}.
\revision{We also provide an estimate in Appendix \ref{appendix_approximation_factors}
below.}

\section{Duality under minimal regularity}
\label{section_duality}

Here, we state in a separate section the key technical result that enables us
to perform duality techniques under minimal regularity assumptions. We believe
it will be important in other contexts, so that we make it easy to reference.
Although the arguments are not complicated, we were not able to find a proof
in the literature.

\begin{lemma}[Lifting error]
\label{lemma_lifting}
For all $\phi \in H^1(\CT_h)$ and $\bsig \in \BX_{\GN}(\ddiv,\Omega)$, we have
\begin{multline}
\label{eq_nonconf_identity}
(\GRAD(\phi),\bsig)_{\CT_h}
+
(\phi,\div \bsig)_{\CT_h}
-
(\phi,\bsig \cdot \bn)_{\CF_h^{\rm R}}
=
\\
(\GRAD(\phi-\widetilde\phi),\bsig-\bsig_h)_{\CT_h}
+
(\phi-\widetilde \phi,\div(\bsig-\bsig_h))_{\CT_h}
-
(\phi-\widetilde \phi,(\bsig-\bsig_h) \cdot \bn)_{\CF_h^{\rm R}}
\end{multline}
and in particular
\begin{equation}
\label{eq_nonconf_estimate}
|(\GRAD(\phi),\bsig)_{\CT_h}+(\phi,\div \bsig)_{\CT_h}-(\phi,\bsig \cdot \bn)_{\CF_h^{\rm R}}|
\leq
\|\phi-\widetilde \phi\|_{\dagger,1,\CT_h}
\|\bsig-\bsig_h\|_{\dagger,\ddiv,\CT_h}
\end{equation}
for all $\widetilde \phi \in H^1_{\GD}(\Omega)$ and
$\bsig_h \in \BCP_p(\CT_h) \cap \BX_{\GN}(\ddiv,\Omega)$.
\end{lemma}

\begin{proof}
We first observe that if $\widetilde \phi \in H^1_{\GD}(\Omega)$, integration by parts gives
\begin{equation*}
(\GRAD(\widetilde \phi),\btau)_{\CT_h}
+
(\widetilde \phi,\div \btau)
-
(\widetilde \phi,\btau \cdot \bn)_{\GR}
=
(\grad \widetilde \phi,\btau)_{\CT_h}
+
(\widetilde \phi,\div \btau)
-
(\widetilde \phi,\btau \cdot \bn)_{\GR}
=
0
\end{equation*}
for all $\btau \in \BX_{\GN}(\ddiv,\Omega)$ due to the essential boundary conditions
on $\GD$ and $\GN$. Thus,
\begin{align*}
(\GRAD(\phi),\bsig)_{\CT_h} + (\phi,\div \bsig)_{\CT_h} - (\phi,\bsig \cdot \bn)_{\CF_h^{\rm R}}
=
(\GRAD(\phi-\widetilde \phi),\bsig)_{\CT_h} + (\phi-\widetilde \phi,\div \bsig)_{\CT_h}
-
(\phi-\widetilde \phi,\bsig \cdot \bn)_{\CF_h^{\rm R}}.
\end{align*}
If $\bsig_h \in \BCP_p(\CT_h) \cap \revision{\BX_{\GN}}(\ddiv,\Omega)$, we also have
\begin{align*}
0
&=
(\GRAD(\phi),\bsig_h)_{\CT_h}
+
(\phi,\div \bsig_h)_{\CT_h}
-
(\phi,\bsig_h \cdot \bn)_{\CF_h^{\rm R}}
\\
&=
(\GRAD(\phi)-\grad \widetilde \phi,\bsig_h)_{\CT_h}
+
(\grad \widetilde \phi,\bsig_h)_{\CT_h}
+
(\phi,\div \bsig_h)_{\CT_h}
-
(\phi,\bsig_h \cdot \bn)_{\CF_h^{\rm R}}
\\
&=
(\GRAD(\phi-\widetilde \phi),\bsig_h)_{\CT_h}
+
(\phi-\widetilde \phi,\div \bsig_h)_{\CT_h}
-
(\phi-\widetilde \phi,\bsig_h \cdot \bn)_{\CF_h^{\rm R}},
\end{align*}
and \eqref{eq_nonconf_identity} follows after summation. Then, \eqref{eq_nonconf_estimate}
is a direct consequence of \eqref{eq_norm_dual}.
\end{proof}

A simple consequence of Lemma \ref{lemma_lifting} is the following result,
which is crucial to estimate the non-conformity error in the context of
duality arguments.

\begin{corollary}[Control of the non-confomity]
\label{corollary_nonconf}
For all $\phi \in H^1(\CT_h)$ and $\xi \in H^1_{\revision{\GD}}(\Omega)$
with $\BA\grad \xi \in \BX_{\GN}(\ddiv,\Omega)$, we have
\begin{equation}
\label{eq_nonconf}
\left |
a_h(\phi,\xi)+(\phi,\div(\BA\grad\xi))_{\CT_h}-(\phi,\BA\grad \xi \cdot \bn)_{\CF_h^{\rm R}}
\right |
\leq
\|\phi-\widetilde \phi\|_{\dagger,1,\CT_h}\|\BA\grad \xi-\bsig_h\|_{\dagger,\ddiv,\CT_h}
\end{equation}
for all $\widetilde \phi \in H^1_{\GD}(\Omega)$ and
$\bsig_h \in \BCP_p(\CT_h) \cap \BX_{\GN}(\ddiv,\Omega)$.
\end{corollary}

\section{A priori analysis}
\label{section_apriori}

The purpose of this section is to establish the existence and uniqueness of
a discrete solution $u_h \in \CP_p(\CT_h)$ to \eqref{eq_helmholtz_ipdg} and
to derive a priori error estimates controlling $u-u_h$ in suitable norms.
The proof follows the line of the Schatz argument \cite{schatz_1974a}, with
suitable modifications to take into account the non-conforming nature of
the discrete scheme.

In this section, we require that there exists $\rho > 0$ such that
\begin{equation}
\label{eq_assumption_stabilization_apriori}
\Re s_h(v_h,v_h) \geq \rho^2 \|v_h-\pig v_h\|_{\dagger,1,\CT_h}^2
\qquad
\forall v_h \in \CP_p(\CT_h).
\end{equation}
This assumption always holds true when
$s_h(\cdot,\cdot) = \beta/h(\jmp{\cdot},\jmp{\cdot})_{\CF_h}$
with $\beta > 0$. It is also the case when $s_h$ takes the form
\eqref{eq_penalization_usual} if the penalization parameter is
sufficiently large \cite{agut_diaz_2013a}. We note that this assumption
rules out some choices of stabilization including complex-value stabilization
parameters, as proposed e.g. in \cite{feng_wu_2009a,sauter_zech_2015a}.
For the sake of simplicity, we only consider stabilizations satisfying
\eqref{eq_assumption_stabilization_apriori}, but slight modifications
of our proofs could handle more general situations.

We start with the main duality argument, which is a generalization of the Aubin-Nitsche trick.

\begin{lemma}[Aubin-Nitsche]
\label{lemma_aubin_nitsche}
Assume that $u_h \in \CP_p(\CT_h)$ satisfies \eqref{eq_helmholtz_ipdg}.
Then, we have
\begin{equation}
\label{eq_aubin_nitsche_omega}
\omega\|u-u_h\|_{\mu,\CT_h}
\leq
\cgbag \enorm{u-u_h}_{\omega,\CT_h}
+
\cgbad \|u_h-\pig u_h\|_{\dagger,1,\CT_h}
\end{equation}
and
\begin{equation}
\label{eq_aubin_nitsche_GR}
\omega^{1/2}\|u-u_h\|_{\gamma,\CF_h^{\rm R}}
\leq
\tgbag \enorm{u-u_h}_{\omega,\CT_h}
+
\tgbad\|u_h-\pig u_h\|_{\dagger,1,\CT_h}.
\end{equation}
\end{lemma}

\begin{proof}
For the sake of shortness, we set $e_h \eq u-u_h$.
We first establish \eqref{eq_aubin_nitsche_omega}. To do so,
we define $\xi$ as the unique element of $H^1_{\GD}(\Omega)$ such that
$b(\phi,\xi) = \omega (\mu\phi,e_h)_\Omega$ for all $\phi \in H^1_{\GD}(\Omega)$.
Integrating by parts, we see that the identities
\begin{equation*}
-\omega^2 \mu \xi - \div\left (\BA\grad \xi\right ) = \revision{\omega\mu} e_h
\qquad
\BA\grad \xi \cdot \bn+i\omega\xi = 0
\end{equation*}
respectively hold in $L^2(\Omega)$ and $L^2(\GR)$. As a result,
\begin{equation*}
\omega \|e_h\|_{\mu,\CT_h}^2
=
-\omega^2 (\mu e_h,\xi)_{\revision{\CT_h}}
-i\omega(\revision{\gamma}e_h,\revision{\xi})_{\CF_h^{\rm R}}
-
\left \{
(e_h,\div(\BA\grad \xi))_{\CT_h}
-(e_h,\BA\grad \xi \cdot \bn)_{\CF^{\rm R}}
\right \}
\end{equation*}
and we have from Corollary \ref{corollary_nonconf} that
\begin{equation*}
\omega \|e_h\|_{\mu,\CT_h}^2
\leq
|b_h(e_h,\xi)|
+
\|e_h-\widetilde \phi\|_{\dagger,1,\CT_h}
\|\BA\grad \xi-\bsig_h\|_{\dagger,\ddiv,\CT_h}
\end{equation*}
for all $\widetilde \phi \in H^1_{\GD}(\Omega)$ and
$\bsig_h \in \BCP_p(\CT_h) \cap \BX_{\GN}(\ddiv,\Omega)$.
On the one hand, we select $\widetilde \phi = u-\pig u_h$ and
$\bsig_h = \pi_h^{\rm d}(\BA\grad\xi)$, so that using \eqref{eq_definition_gbad},
we have
\begin{equation*}
\omega \|e_h\|_{\mu,\CT_h}^2
\leq
|b_h(e_h,\xi)|
+
\cgbad \|u_h-\pig u_h\|_{\dagger,1,\CT_h} \|e_h\|_{\mu,\CT_h}.
\end{equation*}
On the other hand, using \eqref{eq_galerkin_orthogonality} and \eqref{eq_continuity}, we have
\begin{equation*}
|b_h(e_h,\xi)|
=
|b_h(e_h,\xi-\pihg\xi)|
\leq
\enorm{e_h}_{\omega,\CT_h}\enorm{\xi-\pi_h^{\rm g}\xi}_{\omega,\Omega}
\leq
\cgbag \enorm{e_h}_{\omega,\CT_h}\|e_h\|_{\mu,\CT_h}.
\end{equation*}

We then prove \eqref{eq_aubin_nitsche_GR}. Instead of $\xi$,
we define $\chi$ as the unique element of $H^1_{\GD}(\Omega)$
such that $b(\phi,\chi) = \omega^{1/2}(\gamma\revision{\phi},e_h)_{\GR}$ for all
$\phi \in H^1_{\GD}(\Omega)$. We then have
\begin{equation*}
-\omega^2 \mu \chi - \div\left (\BA\grad \chi\right ) = 0
\qquad
\BA\grad \chi \cdot \bn+i\omega\chi = \revision{\omega^{1/2}\gamma}e_h
\end{equation*}
in $L^2(\Omega)$ and $L^2(\GR)$, respectively. We thus have
\begin{equation*}
\omega^{1/2} \|e_h\|_{\gamma,\CF_h^{\rm R}}^2
=
-\omega^2 (\mu e_h,\chi)_{\revision{\CT_h}}
-i\omega(e_h,\phi)_{\CF_h^{\rm R}}
-
\left \{
(e_h,\div(\BA\grad \chi))_{\CT_h}
-(e_h,\BA\grad \chi \cdot \bn)_{\CF^{\rm R}}
\right \},
\end{equation*}
and applying again Corollary \ref{corollary_nonconf} we arrive at
\begin{equation*}
\omega^{1/2} \|e_h\|_{\gamma,\CF_h^{\rm R}}^2
\leq
|b_h(e_h,\chi-\pi_h^{\rm g} \chi)|
+
\|u_h-\pig u_h\|_{\dagger,1,\CT_h}
\|\BA\grad \chi-\pihd(\BA\grad \chi)\|_{\dagger,\ddiv,\CT_h},
\end{equation*}
and the result follows from the definitions of $\tgbag$ and $\tgbad$
given in \eqref{eq_definition_gba_summed}.
\end{proof}

We can now complete the Schatz argument, leading to quasi-optimality
of the discrete solution for sufficiently refined meshes.

\begin{theorem}[A priori estimate]
If $\gbag < 1$ and $\blunder{\sqrt{2}}\gbad \leq \rho$, then there exists a unique
solution $u_h \in \CP_p(\CT_h)$ to \eqref{eq_helmholtz_ipdg} and the estimate
\begin{multline}
\label{eq_aprio_estimate}
\enorm{u-u_h}_{\omega,\CT_h}
\leq
\frac{1}{1-\gbag^2}
\\
\left (
\min_{v_h \in \CP_p(\CT_h) \cap H^1_{\GD}(\Omega)} \enorm{u-v_h}_{\omega,\Omega}^{\blunder{2}}
\blunder{+
\frac{1}{\rho^2}\min_{\bsig_h \in \BCP_p(\CT_h) \cap \BX_{\GN}(\ddiv,\Omega}
\|\BA\grad u-\bsig_h\|_{\dagger,\ddiv,\CT_h}^2}
\right )^{1/2}
\end{multline}
holds true. In addition, if $\blunder{\sqrt{2}}\gbad < \rho$, we also have
\begin{multline}
\label{eq_aprio_nonconf}
\|u_h-\pig u_h\|_{\dagger,1,\CT_h}
\leq
\left (
\frac{1}{(1-\gbag^2)(\rho^2-\blunder{2}\gbad^2)}
\right )^{1/2}
\\
\left (
\min_{v_h \in \CP_p(\CT_h) \cap H^1_{\GD}(\Omega)} \enorm{u-v_h}_{\omega,\Omega}^{\blunder{2}}
\blunder{+
\frac{1}{\rho^2}\min_{\bsig_h \in \BCP_p(\CT_h) \cap \BX_{\GN}(\ddiv,\Omega}
\|\BA\grad u-\bsig_h\|_{\dagger,\ddiv,\CT_h}^2}
\right )^{1/2}.
\end{multline}
\end{theorem}

\blunder
{
\begin{remark}[Higher-order term]
The last term appearing in the right-hand side \eqref{eq_aprio_estimate}
is unpleasant, since in general the finite element error is not solely
controlled by the best-approximation error. Nevertheless, this term
exhibits the right convergence rate. Besides, it can be made higher-order for
sufficiently (piecewise) smooth solutions $u$. Indeed, the polynomial
degree $p$ appearing in the approximation of $\BA\grad u$ only depends
on the polynomial degree chosen for the gradient reconstruction $\GRAD(u_h)$,
not the polynomial degree of the discretization space $\CP_p(\CT_h)$.
It is therefore possible to reconstruct $\GRAD(u_h)$ in, e.g., $\BCP_{p+1}(\CT_h)$
to make the undesired term in \eqref{eq_aprio_estimate} negligible. Notice that
this argument works if the stabilization parameter is sufficiently large according
to \eqref{eq_assumption_stabilization_apriori}.
\end{remark}
}

\begin{proof}
We first observe that
\begin{align*}
\Re b_h(e_h,e_h) + \omega\|e_h\|_{\gamma,\CF_h^{\rm R}}^2 + 2\omega^2\|e_h\|_{\mu,\CT_h}^2
&=
\enorm{e_h}_{\omega,\CT_h}^2
+
\Re s_h(u_h,u_h)
\\
&\geq
\enorm{e_h}_{\omega,\revision{\CT_h}}^2 + \rho^2 \|u_h-\pig u_h\|_{\dagger,1,\CT_h}^2.
\end{align*}

\blunder
{
On the other hand, \revision{for $v_h \in \CP_p(\CT_h) \cap H^1_{\GD}(\Omega)$}, we have
\begin{equation*}
b_h(e_h,e_h)
=
b_h(e_h,u-v_h)+b_h(e_h,v_h-u_h)
=
b_h(e,u-v_h) + b_h(u,v_h-u_h)-(f,v_h-u_h)_{\CT_h}.
\end{equation*}
Since $a_h$ is Hermitian, Corollary \ref{corollary_nonconf} ensures that
\begin{align*}
|b_h(u,v_h-u_h)-(f,v_h-u_h)_{\CT_h}|
&\leq
|a_h(u,v_h-u_h)-(\div(\BA\grad u),v_h-u_h)_{\CT_h}|
\\
&\leq
\|u_h-\pi_h^{\rm g}u_h\|_{\dagger,1,\CT_h}
\min_{\bsig_h \in \BCP_p(\CT_h) \cap \BX_{\GN}(\ddiv,\Omega}
\|\BA\grad u-\bsig_h\|_{\dagger,\ddiv,\CT_h}.
\\
&\leq
\frac{\rho^2}{2}\|u_h-\pi_h^{\rm g}u_h\|_{\dagger,1,\CT_h}^2
+
\frac{1}{2\rho^2}\min_{\bsig_h \in \BCP_p(\CT_h) \cap \BX_{\GN}(\ddiv,\Omega}
\|\BA\grad u-\bsig_h\|_{\dagger,\ddiv,\CT_h}^2.
\end{align*}
}

As a result, we have
\begin{align*}
\enorm{e_h}_{\omega,\CT_h}^2
+
\blunder{\frac{1}{2}}
\rho^2 \|u_h-\pig u_h\|_{\dagger,1,\CT_h}^2
&\leq
|b_h(e_h,u-v_h)| + 2\omega^2\|e_h\|_{\mu,\CT_h}^2 + \omega\|e_h\|_{\gamma,\CF_h^{\rm R}}^2
\\
&\blunder{+
\frac{1}{2\rho^2}\min_{\bsig_h \in \BCP_p(\CT_h) \cap \BX_{\GN}(\ddiv,\Omega}
\|\BA\grad u-\bsig_h\|_{\dagger,\ddiv,\CT_h}^2}
\\
&\leq
\enorm{e_h}_{\omega,\CT_h} \enorm{u-v_h}_{\omega,\Omega}
\\
&+
2\left (\cgbag \enorm{e_h}_{\omega,\CT_h} + \cgbad \|u_h-\pig u_h\|_{\dagger,1,\CT_h}\right )^2
\\
&+
\left (\tgbag \enorm{e_h}_{\omega,\CT_h} + \tgbad \|u_h-\pig u_h\|_{\dagger,1,\CT_h}\right )^2
\\
&
\blunder{+
\frac{1}{2\rho^2}\min_{\bsig_h \in \BCP_p(\CT_h) \cap \BX_{\GN}(\ddiv,\Omega}
\|\BA\grad u-\bsig_h\|_{\dagger,\ddiv,\CT_h}^2}
\\
&\leq
\enorm{e_h}_{\omega,\CT_h} \enorm{u-v_h}_{\omega,\Omega}
\\
&+
(4\cgbag+2\tgbag) \enorm{e_h}_{\omega,\CT_h}^2
+
(4\cgbad+2\tgbad) \|u_h-\pig u_h\|_{\dagger,1,\CT_h}^2
\\
&
\blunder{+
\frac{1}{2\rho^2}\min_{\bsig_h \in \BCP_p(\CT_h) \cap \BX_{\GN}(\ddiv,\Omega}
\|\BA\grad u-\bsig_h\|_{\dagger,\ddiv,\CT_h}^2}
\end{align*}
so that
\begin{multline*}
(1-4\cgbag^2-2\tgbag^2)\enorm{e_h}_{\omega,\CT_h}^2
+
\blunder{\frac{1}{2}}
(\rho^2-\blunder{8}\cgbad^2-\blunder{4}\tgbad^2)\|u_h-\pi_h^{\rm g} u_h\|_{\dagger,1,\CT_h}^2
\\
\leq
\enorm{e_h}_{\omega,\CT_h}\enorm{u-v_h}_{\omega,\Omega}
\blunder{+
\frac{1}{2\rho^2}\min_{\bsig_h \in \BCP_p(\CT_h) \cap \BX_{\GN}(\ddiv,\Omega}
\|\BA\grad u-\bsig_h\|_{\dagger,\ddiv,\CT_h}^2}
\end{multline*}
\blunder{and
\begin{multline*}
(1-4\cgbag^2-2\tgbag^2)\enorm{e_h}_{\omega,\CT_h}^2
+
\blunder
(\rho^2-\blunder{8}\cgbad^2-\blunder{4}\tgbad^2)\|u_h-\pi_h^{\rm g} u_h\|_{\dagger,1,\CT_h}^2
\\
\leq
\enorm{e_h}^2
+
\frac{1}{\rho^2}\min_{\bsig_h \in \BCP_p(\CT_h) \cap \BX_{\GN}(\ddiv,\Omega}
\|\BA\grad u-\bsig_h\|_{\dagger,\ddiv,\CT_h}^2.
\end{multline*}
At this point,}
\eqref{eq_aprio_estimate} and \eqref{eq_aprio_nonconf} follow
recalling the definitions of $\gbag$ and $\gbad$ in \eqref{eq_definition_gba}.
\end{proof}

%%%   \begin{corollary}[]
%%%   We have
%%%   \begin{equation*}
%%%   \jay_h(u_h,u_h)
%%%   \leq
%%%   \left (\frac{2}{1-2\cgbag^2}\right )^{1/2}
%%%   \min_{v_h \in \CP_p(\CT_h) \cap H^1_0(\Omega)} \enorm{u-v_h}_{\omega,\Omega}.
%%%   \end{equation*}
%%%   \end{corollary}
%%%   
%%%   \begin{proof}
%%%   \begin{align*}
%%%   \jay_h(u_h,u_h)
%%%   &=
%%%   b_h(e_h,e_h) + \omega^2\|e_h\|_{\mu,\Omega} - \|\GRAD(e_h)\|_{\BA,\Omega}^2
%%%   \\
%%%   &=
%%%   b_h(e_h,u-v_h) + \omega^2\|e_h\|_{\mu,\Omega} - \|\GRAD(e_h)\|_{\BA,\Omega}^2
%%%   \\
%%%   &\leq
%%%   \enorm{e_h}_{\omega,\CT_h}\enorm{u-v_h}_{\omega,\Omega}
%%%   +
%%%   \enorm{e_h}_{\omega,\CT_h}^2
%%%   \\
%%%   &\leq
%%%   \left \{
%%%   \frac{1}{1-2\cgbag^2}
%%%   +
%%%   \left (\frac{1}{1-2\cgbag^2}\right )^2
%%%   \right \}
%%%   \enorm{u-v_h}_{\omega,\Omega}^2
%%%   \\
%%%   &\leq
%%%   2
%%%   \left (\frac{1}{1-2\cgbag^2}\right )^2
%%%   \enorm{u-v_h}_{\omega,\Omega}^2.
%%%   \end{align*}
%%%   \end{proof}

\section{A posteriori analysis}
\label{section_aposteriori}

We now provide a posteriori error estimates under minimal regularity assumptions.
We first present an abstract framework, and then show how it can be applied to the
particular of residual-based estimators.

Throughout this section, we assume that $u_h$ is a fixed function in
$\CP_p(\CT_h)$ satisfying \eqref{eq_helmholtz_ipdg}. Notice that unique
solvability of \eqref{eq_helmholtz_ipdg} is not required. In particular,
the proposed analysis applies without any restriction on the mesh size
or polynomial degree. Also, in contrast to the a priori analysis, we do
not need any specific positivity assumption on the stabilisation form $s_h$.

\subsection{Abstract reliability analysis}

\begin{subequations}
Following \cite[Theorem 3.3]{ern_vohralik_2015a}, the discretisation error
may be bounded using two different terms that respectively control the equation
residual, and the non-conformity of the discrete solution. For the equation
residual, we introduce the residual functional and its norm
\begin{equation}
\label{eq_definition_Rr}
\langle \LR_{\rm r}, v\rangle \eq b_h(e_h,v) \quad \forall v \in H^1_{\GD}(\Omega),
\qquad
\TR_{\rm r}
\eq
\sup_{\substack{v \in H^1_{\GD}(\Omega) \\ \|\grad v\|_{\BA,\Omega} = 1}}
|\langle \LR_{\rm r},v\rangle|.
\end{equation}
Notice that this first term only involves conforming test functions.
To account for the non-conformity of the discrete solution, we additionally
introduce
\begin{equation}
\label{eq_definition_Rc}
\TR_{\rm c} \eq \|u_h-\pig u_h\|_{\dagger,1,\CT_h}.
\end{equation}
The total abstract estimator is then the Hilbertian sum of the two above components:
\begin{equation}
\label{eq_definition_R}
\TR^2 \eq \TR_{\rm r}^2 + \TR_{\rm c}^2.
\end{equation}
\end{subequations}

Following \cite{chaumontfrelet_ern_vohralik_2021a,sauter_zech_2015a},
our a posteriori analysis follows the lines of the Schatz argument \cite{schatz_1974a}.
Specifically, we start by controlling weak norms of the errors with
the estimator and the approximation factor using a duality argument.

\begin{lemma}[Aubin-Nitsche]
The following estimates hold true:
\begin{equation}
\label{eq_aubin_nitsche_apost}
\omega \|u-u_h\|_{\mu,\CT_h} \leq \cgba \TR,
\qquad
\omega^{1/2} \|u-u_h\|_{\gamma,\CF_h^{\rm R}} \leq \tgba \TR.
\end{equation}
\end{lemma}

\begin{proof}
Let us set $e_h \eq u-u_h$.
We start with the first estimate in \eqref{eq_aubin_nitsche_apost}.
We define $\xi$ as the unique element of $H^1_{\GD}(\Omega)$ such that
$b(w,\xi) = \omega (\mu w,e_h)_{\CT_h}$ for all $w \in H^1_{\GD}(\Omega)$.
Following the lines the proof of Lemma \ref{lemma_aubin_nitsche},
we write
\begin{align*}
b_h(e_h,\xi)
&=
-\omega^2(\mu e_h,\xi)_{\CT_h}
-i\omega(\gamma e_h,\xi)_{\CF_h^{\rm R}}
+(\GRAD(e_h),\BA\grad \xi)_{\CT_h}
\\
&=
(e_h,-\omega^2 \mu\xi-\div(\BA\grad \xi))_{\CT_h}
-(e_h,\BA\grad\xi\cdot\bn)_{\CF_h^{\rm R}}+
(e_h,\div(\BA\grad \xi))_{\revision{\CT_h}}+(\GRAD(e_h),\BA\grad \xi)_{\CT_h}
\\
&=
\omega \|e_h\|_{\mu,\CT_h}^2
+
(\GRAD(\pig u_h -u_h),\revision{\BA\grad\xi}-\bsig_h)_{\CT_h}
\\
&+
(\pig u_h-u_h,\div(\revision{\BA\grad\xi}-\bsig_h))_{\CT_h}
\revision{
-
(\pig u_h-u_h,(\BA\grad\xi-\bsig_h)\cdot \bn)_{\CF_h^{\rm R}}
},
\end{align*}
for all $\bsig_h \in \BCP_p(\CT_h) \cap \revision{\BX_{\GN}}(\ddiv,\Omega)$,
where we employed \eqref{eq_nonconf_identity} in the last identity. It follows that
\begin{equation*}
\omega \|e_h\|_{\mu,\CT_h}^2
\leq
|b_h(e_h,\xi)| + \|u_h-\pig u_h\|_{\dagger,1,\CT_h}\|\BA\grad \xi-\bsig_h\|_{\dagger,\ddiv,\CT_h}.
\end{equation*}
On the one hand, recalling the Galerkin orthogonality property stated
in \eqref{eq_galerkin_orthogonality} and the definition of $\TR_{\rm r}$
at \eqref{eq_definition_Rr}, we have
\begin{equation*}
|b_h(e_h,\xi)|
=
|b_h(e_h,\xi-\pi_h^{\rm g}\xi)|
\leq
\TR_{\rm r} \|\grad(\xi-\pi_h^{\rm g}\xi)\|_{\BA,\Omega}
\leq
\revision{\cgbag} \TR_{\rm r} \|e_h\|_{\mu,\CT_h}.
\end{equation*}
On the other hand, recalling \eqref{eq_definition_gbad} and \eqref{eq_definition_Rc}, we have
\begin{equation*}
\|u_h-\pig u_h\|_{\dagger,1,\CT_h}\|\BA\grad \xi-\pi_h^{\rm d}(\BA\grad \xi)\|_{\dagger,\ddiv,\CT_h}
\leq
\revision{\cgbad}\TR_{\rm c}\|e_h\|_{\mu,\CT_h}.
\end{equation*}
Then, the first estimate of \eqref{eq_aubin_nitsche_apost} follows from \eqref{eq_definition_R}
since
\begin{equation*}
\revision{\cgbag} \TR_{\rm r} + \revision{\cgbad} \TR_{\rm c}
\leq
\left (
\revision{\cgbag^{{\color{black} 2}}} + \revision{\cgbad^{{\color{black} 2}}}
\right )^{1/2}
\left (
\TR_{\rm r}^2
+
\TR_{\rm c}^2
\right )^{1/2}
=
\cgba \TR.
\end{equation*}

For the sake of shortness, we do not detail the proof of the second estimate
in \eqref{eq_aubin_nitsche_apost} as it follows the lines of the first estimate
but using the function $\chi$ from the proof of Lemma \ref{lemma_aubin_nitsche}
instead of $\xi$.
\end{proof}

We can now conclude the abstract reliability analysis.
Following \cite{ern_vohralik_2015a}, our proof uses a Pythagorean
identity to separate the conforming and non-conforming parts of the error.

\begin{theorem}[Abstract reliability]
We have
\begin{equation}
\label{eq_abstract_reliability}
\enorm{u-u_h}_{\omega,\CT_h}
\leq
\sqrt{1+\gba^2} \TR.
\end{equation}
\end{theorem}

\begin{proof}
For the sake of simplicity, we introduce
\begin{equation*}
b^+_h(\phi,v)
\eq
\omega^2 (\mu\phi,v)_{\CT_h}
+
\omega (\gamma\phi,v)_{\CF_h^{\rm R}}
+
(\BA\GRAD(\phi),\GRAD(v))_{\CT_h}
\qquad
\forall \phi,v \in H^1(\CT_h),
\end{equation*}
and we define $\widetilde u$ as the only element of $H^1_{\GD}(\Omega)$ such that
\begin{equation*}
b^+_h(\widetilde u-u_h,v) = 0 \qquad \forall v \in H^1_{\GD}(\Omega).
\end{equation*}
Notice that this indeed a well-formed definition, since $b^+_h$ is equivalent
to the usual inner-product of $H^1_{\GD}(\Omega)$. Classically, we have the Pythagorean identity
\begin{align}
\nonumber
\enorm{u-u_h}_{\omega,\CT_h}^2
&=
\enorm{\widetilde u-u_h}_{\omega,\CT_h}^2
+
2\Re b^+(\widetilde u - u_h,u-\widetilde u)
+
\enorm{u-\widetilde u}_{\omega,\CT_h}^2
\\
\label{tmp_rel1}
&=
\enorm{u_h-\widetilde u}_{\omega,\CT_h}^2
+
\enorm{u-\widetilde u}_{\omega,\CT_h}^2.
\end{align}
Then, on the one hand, it is clear using \eqref{eq_norm_control} that
\begin{equation}
\label{tmp_rel2}
\enorm{u_h-\widetilde u}_{\omega,\CT_h}
\leq
\enorm{u_h-\pig u_h}_{\omega,\CT_h}
\leq
\|u_h-\pig u_h\|_{\dagger,1,\CT_h}
=
\TR_{\rm c},
\end{equation}
and on the other hand, we have
\begin{align*}
\enorm{u-\widetilde u}_{\omega,\Omega}^2
&=
b^+_h(u-\widetilde u,u-\widetilde u)
=
b^+_h(e_h,u-\widetilde u)
\\
&=
b_h(e_h,u-\widetilde u)
+
2\omega^2(\mu e_h,u-\widetilde u)_{\CT_h}
+
(1-i)\omega (\gamma e_h,u-\widetilde u)_{\CF_h^{\rm R}}
\\
&\leq
\TR_{\rm r}
\|\grad(u-\widetilde u)\|_{\BA,\Omega}
+
2\cgba \TR \omega\|u-\widetilde u\|_{\mu,\CT_h}
+
\sqrt{2}\tgba \TR \omega^{\revision{1/2}}\|u-\widetilde u\|_{\gamma,\CF_h^{\rm R}}
\\
&\leq
\left (
\TR_{\rm r}^2
+
(4\cgba^2 + 2\tgba^2) \TR^2
\right )^{1/2}
\enorm{u-\widetilde u}_{\omega,\Omega},
\end{align*}
and recalling \eqref{eq_definition_gba_summed}, we arrive at
\begin{equation}
\label{tmp_rel3}
\enorm{u-\widetilde u}_{\omega,\Omega}^2
\leq
\TR_{\rm r}^2
+
(4\cgba^2+2\tgba^2) \TR^2
=
\TR_{\rm r}^2
+
\gba^2 \TR^2.
\end{equation}
The estimate in \eqref{eq_abstract_reliability} then follows from
\eqref{tmp_rel1}, \eqref{tmp_rel2} and \eqref{tmp_rel3}.
\end{proof}

\subsection{Residual-based a posteriori estimator}

The residuals $\TR_{\rm r}$ and $\TR_{\rm c}$ can be straightforwardly
controlled using flux and potential reconstructions. We refer the reader
to \cite{congreve_gedicke_perugia_2019a,ern_vohralik_2015a} for details
about these constructions. Here, we focus instead on residual-based
estimators, for which the link may be less clear. In this section, the
letter $C$ refers to a generic constant that may change from one occurrence
to the other, and that only depends on the contrasts of the coefficients
$\mu$, $\BA$ and $\gamma$ as well as the polynomial degree $p$ and the
shape-regularity parameter $\kappa$.
The stability result for the lifting operator
\begin{equation}
\label{eq_stability_lifting}
\|\LIFT{v_h}\|_{\BA,\CT_h}^2
\leq
C
\sum_{K \in \CT_h} \frac{\alpha_K}{h_K} \left \|\jmp{v_h}\right \|_{\partial K}^2
\quad
\forall v_h \in \CP_p(\CT_h)
\end{equation}
can be found, e.g., in \cite[Section 4.3]{dipietro_ern_2012a}
or \cite[Lemma 4.1]{houston_schotzau_wihler_2006a}.

We first establish that $\TR_{\rm c}$ can be controlled by the jumps of $u_h$. The notations
\begin{equation*}
\frac{\omega h_F^\star}{\vel_F^\star}
\eq
\max_{F \in \CF_h^{\rm R}} \frac{\omega h_F}{\vel_F},
\qquad
\frac{\omega h_K^\star}{\vel_K^\star}
\eq
\max_{K \in \CT_h} \frac{\omega h_K}{\vel_K},
\end{equation*}
will be useful.

\begin{lemma}[Control of the non-conformity]
We have
\begin{equation}
\label{eq_estimate_Rc}
\TR_{\rm c}^2
\leq
C
\max \left (
1,\frac{\omega h_F^\star}{\vel_F^\star},\left (\frac{\omega h_K^\star}{\vel_K^\star}\right )^2
\right )
\sum_{K \in \CT_h} \frac{\alpha_K}{h_K} \|\jmp{u_h}\|_{\partial K \setminus (\GN \cup \GR)}^2.
\end{equation}
\end{lemma}

\begin{proof}
We start with two estimates that are easily inferred from
\cite[Lemma 4.3]{ern_guermond_2017a}. For all $v_h \in \CP_p(\CT_h)$,
there exists a conforming approximation $\CJ v_h \in \CP_p(\CT_h) \cap H^1_{\GD}(\Omega)$
such that and all $K \in \CT_h$, we have
\begin{equation*}
\|v_h-\CJ v_h\|_K^2
\leq
C h_K \sum_{K' \in \widetilde \CT_K} \|\jmp{v_h}\|_{\partial K'}^2,
\qquad
\|\grad(v_h-\CJ v_h)\|_K^2
\leq
C \frac{1}{h_K} \sum_{K' \in \widetilde \CT_K} \|\jmp{v_h}\|_{\partial K'}^2,
\end{equation*}
where $\widetilde \CT_K$ collects those elements $K' \in \CT_h$ sharing at least
one vertex with $K$. Employing the multiplicative trace inequality
\begin{equation*}
\|\theta\|_{\partial K}^2
\leq
C \left \{
\frac{1}{h_K}\|\theta\|_K^2 + \|\theta\|_K\|\grad \theta\|_K
\right \}
\leq
C \left \{
\frac{1}{h_K}\|\theta\|_K^2 + h_K\|\grad \theta\|_K^2
\right \}
\qquad
\forall \theta \in H^1(K),
\end{equation*}
we also have
\begin{equation*}
\|v_h-\CJ v_h\|_{\partial K}^2
\leq
C \sum_{K' \in \widetilde \CT_K} \|\jmp{v_h}\|_{\partial K'}^2.
\end{equation*}
It is therefore clear that
\begin{equation*}
\|u_h-\CJ u_h\|_{\dagger,1,\CT_h}^2
\leq
C \max \left (
1,\frac{\omega h_F^\star}{\vel_F^\star},\left (\frac{\omega h_K^\star}{\vel_K^\star}\right )^2
\right )
\sum_{K \in \CT_h} \frac{\alpha_K}{h_K} \|\jmp{u_h}\|_{\partial K}^2,
\end{equation*}
and \eqref{eq_estimate_Rc} follows from the definition of $\TR_{\rm c}$ in \eqref{eq_definition_Rc}.
\end{proof}

We then control the conforming residual $\TR_{\rm r}$.
This is classically done by combining element-wise integration
by parts and a quasi-interpolation operator.

\begin{lemma}[Control of the residual]
We have
\begin{multline}
\label{eq_estimate_Rr}
\TR_{\rm r}^2
\leq
C
\Bigg \{
\sum_{K \in \CT_h}
\left (
\frac{h_K^2}{\alpha_K}
\|f+\omega^2 \mu u_h + \div (\BA\grad u_h)\|_{\CT_h}^2
+
\frac{h_K}{\alpha_K}
\|\jmp{\BA\grad u_h} \cdot \bn_K\|_{\partial K \setminus \partial \Omega}^2
+
\frac{\alpha_K}{h_K}
\|\jmp{u_h}\|_{\partial K \setminus \partial \Omega}^2
\right )
\\
+\sum_{F \in \CF_h^{\rm D}} \frac{\alpha_F}{h_F} \|u_h\|_F^2
+\sum_{F \in \CF_h^{\rm N}} \frac{h_F}{\alpha_F} \|\BA \grad u_h \cdot \bn\|_F^2
+\sum_{F \in \CF_h^{\rm R}} \frac{h_F}{\alpha_F} \|\BA\grad u_h \cdot \bn-i\omega\gamma u_h\|_F^2
\Bigg \}.
\end{multline}
\end{lemma}

\begin{proof}
We consider an arbitrary $v \in H^1_{\GD}(\Omega)$, and let
$\widetilde v \eq v-\CQ_h v \in H^1_{\GD}(\Omega)$, where $\CQ_h$
is a quasi-interpolation operator satisfying
\begin{equation*}
\frac{1}{h_K^2} \|v-\CQ_h v\|_K^2
+
\frac{1}{h_K}\|v-\CQ_h\revision{v}\|_{\partial K}^2
+
\|\grad(v-\CQ_h v)\|_K^2
\leq
C \frac{1}{\alpha_K} \|\grad v\|_{\BA,\widetilde K}^{\revision{2}},
\end{equation*}
with $\widetilde K$ the domain corresponding to all the elements $K' \in \CT_h$
that share at least one vertex with $K$. We refer the reader to, e.g.,
\cite[Theorem 5.2]{ern_guermond_2017a}, for the construction of $\CQ_h$.
Thanks to Galerkin orthogonality \eqref{eq_galerkin_orthogonality}, we
can write that
\begin{align}
\nonumber
\langle \LR_{\rm r}, v \rangle
&=
b_h(e_h,v)
=
b_h(e_h,\widetilde v)
\\
\nonumber
&=
(f,\widetilde v)_\Omega
+
\omega^2 (\mu u_h,\widetilde v)_{\CT_h}
+
i\omega (\gamma u_h,\widetilde v)_{\CF_h^{\rm R}}
\revision{-}
(\BA\GRAD(u_h),\grad \widetilde v)_{\CT_h}
\\
\label{tmp_rel_Rc0}
&=
\underbrace{(f,\widetilde v)_\Omega
+
\omega^2 (\mu u_h,\widetilde v)_{\CT_h}
+
i\omega (\gamma u_h,\widetilde v)_{\CF_h^{\rm R}}
-
(\BA\grad u_h,\grad \widetilde v)_{\CT_h}
}_{r_1}
\revision{+}
\underbrace{(\BA\LIFT{u_h},\grad \widetilde v)_{\CT_h}}_{r_2}.
\end{align}
Then, on the one hand, we see that
\begin{align*}
r_1
&=
(f+\omega^2\mu u_h+\div(\BA\grad u_h),\widetilde v)_{\CT_h}
+
i\omega(\gamma u_h,\widetilde v)_{\CF_h^{\rm R}}
-
\sum_{K \in \CT_h} (\BA \grad u_h \cdot \bn_K,\widetilde v)_{\partial K}
\\
&=
(f+\omega^2\mu u_h\revision{+}\div(\BA\grad u_h),\widetilde v)_{\CT_h}
\revision{-}
\sum_{F \in \CF_h^{\rm i}} (\jmp{\BA \grad u_h} \cdot \bn_F,\widetilde v)_{F}
\\
&
-(\BA \grad u_h \cdot \bn,\widetilde v)_{\CF_h^{\rm N}}
-(\BA \grad u_h \cdot \bn - i\omega\gamma u_h,\widetilde v)_{\CF_h^{\rm R}},
\end{align*}
and therefore,
\begin{align}
\nonumber
|r_1|
&\leq
\sum_{K \in \CT_h}
\left (
\|f+\omega^2\mu u_h+\div\left(\BA\grad u_h\right )\|_K
\|\widetilde v\|_K
+
\|\jmp{\revision{\BA}\grad u_h}\cdot \bn_K\|_{\partial K}\|\widetilde v\|_{\partial K}
\right )
\\
\nonumber
&
+\sum_{F \in \CF_h^{\rm N}}\|\BA\grad u_h\cdot \bn\|_F\|\widetilde v\|_F
+\sum_{F \in \CF_h^{\rm R}}\|\BA\grad u_h\cdot \bn-i\omega\gamma u_h\|_F\|\widetilde v\|_F
\\
\nonumber
&\leq
C \Bigg \{
\sum_{K \in \CT_h}
\left (
\frac{h_K^2}{\alpha_K} \|f+\omega^2\mu u_h+\div\left(\BA\grad u_h\right )\|_K^{\revision{2}}
+
\frac{h_K}{\alpha_K}\|\jmp{\grad u_h}\cdot \bn_K\|_{\partial K}^{\revision{2}}
\right )
\\
\label{tmp_rel_Rc1}
&
\qquad
+\sum_{F \in \CF_h^{\rm N}}\frac{h_F}{\alpha_F}\|\BA\grad u_h\cdot \bn\|_F^{\revision{2}}
+\sum_{F \in \CF_h^{\rm R}}\frac{h_F}{\alpha_F}\|\BA\grad u_h\cdot \bn-i\omega\gamma u_h\|_F^{\revision{2}}
\Bigg \}^{\revision{1/2}} \|\grad v\|_{\BA,\Omega}.
\end{align}
On the other hand, we have
\begin{equation}
\label{tmp_rel_Rc2}
|r_2|
\leq
\|\LIFT{u_h}\|_{\BA,\Omega}\|\grad \widetilde v\|_{\BA,\Omega}
\leq
C \left \{
\sum_{K \in \CT_h} \frac{\alpha_K}{h_K} \|\jmp{u_h}\|_{\partial K \setminus \partial \Omega}^2
+
\sum_{F \in \CF_h^{\rm D}} \frac{\alpha_F}{h_F} \|\jmp{u_h}\|_F^2
\right \}^{\revision{1/2}}
\|\grad v\|_{\BA,\Omega},
\end{equation}
and \eqref{eq_estimate_Rr} follows from \eqref{tmp_rel_Rc0}, \eqref{tmp_rel_Rc1}
and \eqref{tmp_rel_Rc2} recalling the definition of $\TR_{\rm r}$ in \eqref{eq_definition_Rr}.
\end{proof}

We therefore define a residual-based estimator by
\begin{align*}
\eta^2
&\eq
\sum_{K \in \CT_h}
\left (
\frac{h_K^2}{\alpha_K}
\|f+\omega^2 \mu u_h \revision{+} \div (\BA\grad u_h)\|_{\CT_h}^2
+
\frac{h_K}{\alpha_K}
\|\jmp{\BA\grad u_h} \cdot \bn_K\|_{\partial K \setminus \partial \Omega}^2
+
\frac{\alpha_K}{h_K}
\|\jmp{u_h}\|_{\partial K \setminus \partial \Omega}^2
\right )
\\
&
+\sum_{F \in \CF_h^{\rm D}} \frac{\alpha_F}{h_F} \|u_h\|_F^2
+\sum_{F \in \CF_h^{\rm N}} \frac{h_F}{\alpha_F} \|\BA \grad u_h \cdot \bn\|_F^2
+\sum_{F \in \CF_h^{\rm R}} \frac{h_F}{\alpha_F} \|\BA\grad u_h \cdot \bn\revision{-i\omega\gamma u_h}\|_F^2.
\end{align*}
As a direct consequence of \eqref{eq_abstract_reliability},
\eqref{eq_estimate_Rc} and \eqref{eq_estimate_Rr}, we obtain a
reliability estimate stated in Theorem \ref{theorem_reliability_residual}.
It is to be compared with \cite[Theorem 2.3]{chaumontfrelet_ern_vohralik_2021a}
and \cite[Theorem 3.6]{sauter_zech_2015a}.

\begin{theorem}[Reliability of the residual estimator]
\label{theorem_reliability_residual}
We have
\begin{equation*}
\enorm{e_h}_{\omega,\CT_h}
\leq
C
\revision{
\left (
1
+
\left (\frac{\omega h_F^\star}{\vel_F^\star}\right )^{1/2}
+
\frac{\omega h_K^\star}{\vel_K^\star}
\right )}
(1+\gba) \eta.
\end{equation*}
\end{theorem}

\subsection{Efficiency}

For the sake of \revision{completeness}, we \revision{state an efficiency estimate
for the residual-based estimator. The proof can be found in
\cite[Section 3.3]{sauter_zech_2015a}, and we also refer the reader to
\cite[Section 4]{congreve_gedicke_perugia_2019a}
for the corresponding result for equilibrated estimators.

Considering an element $K \in \CT_h$, we need a few additional notation.
The component of the error estimator associated with $K$ reads
\begin{align*}
\eta_K^2
&\eq
\frac{h_K^2}{\alpha_K}
\|f+\omega^2 \mu u_h \revision{+} \div (\BA\grad u_h)\|_{\CT_h}^2
+
\frac{h_K}{\alpha_K}
\|\jmp{\BA\grad u_h} \cdot \bn_K\|_{\partial K \setminus \partial \Omega}^2
+
\frac{\alpha_K}{h_K}
\|\jmp{u_h}\|_{\partial K \setminus \partial \Omega}^2
\\
&
+\sum_{\substack{F \in \CF_h^{\rm D} \\ F \subset \partial K}} \frac{\alpha_F}{h_F} \|u_h\|_F^2
+\sum_{\substack{F \in \CF_h^{\rm N} \\ F \subset \partial K}} \frac{h_F}{\alpha_F} \|\BA \grad u_h \cdot \bn\|_F^2
+\sum_{\substack{F \in \CF_h^{\rm R} \\ F \subset \partial K}} \frac{h_F}{\alpha_F} \|\BA\grad u_h \cdot \bn\revision{-i\omega\gamma u_h}\|_F^2.
\end{align*}
Then, to measure the discretization error locally around $K$, we introduce
\begin{equation*}
\enorm{u-u_h}_{\omega,\widetilde K}^2
=
\sum_{\substack{K' \in \CT_h \\ \partial K' \cap \partial K \neq \emptyset}}
\left \{
\omega^2 \|u-u_h\|_{\mu,K'}^2
+
\omega\|u-u_h\|_{\gamma,\GR \cap \partial K'}^2
+
\|\grad(u-u_h)\|_{\BA,K'}^2
\right \}.
\end{equation*}
Finally, the quantity
\begin{equation*}
\osc_{\widetilde K}^2
\eq
\sum_{\substack{K' \in \CT_h \\ \partial K' \cap \partial K \neq \emptyset}}
h_{K'}^2 \min_{f_p \in \CP_p(K')} \|f-f_p\|_{\mu,K'}^2
\end{equation*}
is usually called the ``data-oscillation term''. It is of higher-order
if $f$ is piecewise smooth.

\begin{theorem}[Efficiency of the residual estimator]
For all $K \in \CT_h$, we have
\begin{equation*}
\eta_K
\leq
C
\left \{
\left (
1
+
\max_{\substack{F \in \CF_h^{\rm R} \\ F \subset \partial K}}
\left ( \frac{\omega h_F}{\vel_F} \right )^{1/2}
+
\frac{\omega h_K}{\vel_K}
\right )
\enorm{u-u_h}_{\omega,\widetilde K}
+
\osc_{\widetilde K}
\right \}.
\end{equation*}
\end{theorem}
}

\bibliographystyle{amsplain}
\bibliography{bibliography.bib}

\appendix

\revision
{
\section{Approximation factors}
\label{appendix_approximation_factors}

In this section, we provide an estimate for the approximation factor $\gba$
valid under minimal regularity assumptions. For the sake of simplicity and
shortness, we restrict our attention to the case where $\GR = \emptyset$.
The reason behind this choice is twofold. Obviously, we automatically have
$\tgbag = \tgbad = 0$ in this case, since there are no boundary right-hand
sides. But more importantly, the norm $\|{\cdot}\|_{\dagger,\ddiv,\CT_h}$
does not contain a boundary term. This allows us to easily use the quasi-interpolation
operators introduced in \cite{ern_gudi_smears_vohralik_2022a} to upper-bound $\gbad$
(see the estimates in \eqref{eq_interpolation} below). Notice that, in principle,
there is no obstruction to the construction of quasi-interpolation operators with
good interpolation properties on the boundary if they are defined on $\BX_{\GN}(\ddiv,\Omega)$.
However, the author is not aware of any convenient reference where such an operator can be found.
Alternatively, the canonical Raviart--Thomas interpolation operator has optimal approximation
properties on $\GR$, but it is not defined under minimal regularity (see e.g.
\cite{chaumontfrelet_2019a}).

Classically, the analysis of this section relies on (arbitrarily low) elliptic regularity shifts.
To properly state them, we need to introduce (broken) fractional Sobolev norms
\cite{adams_fournier_2003a}. First, for $K \in \CT_h$, $\bw \in \BL^2(\Omega)$,
and $s \in (0,1)$, we set
\begin{equation*}
|\bw|_{H^s(K)}^2
\eq
\int_K \int_K
\frac{|\bw(\bx)-\bw(\by)|^2}{|\bx-\by|^{2s+d}}d\by d\bx
\end{equation*}
and, for $s = 1$,
\begin{equation*}
|\bw|_{H^1(K)}^2 \eq \sum_{n=1}^d \|\grad \bw_n\|_K^2.
\end{equation*}
Then, the corresponding global norm reads
\begin{equation*}
|\bw|_{H^s(\CT_h)}^2
\eq
\sum_{K \in \CT_h} |\bw|_{H^s(K)}^2.
\end{equation*}

There exists $\theta \in (0,1]$ and such that such that all $v \in H^1_{\GD}(\Omega)$ with
$\BA\grad v \in \BX_{\GN}(\Omega)$, we have
\begin{equation}
\label{eq_elliptic_regularity}
|\BA\grad v|_{H^\theta(\CT_h)} + \alpha_{\max} |\grad v|_{H^\theta(\CT_h)}
\leq
C \ell^{1-\theta} \|\div(\BA\grad v)\|_\Omega,
\end{equation}
where $\ell$ is the diameter of $\Omega$, see e.g. \cite{jochamnn_1999a}.
For generic values of $\BA$, $\theta$ may be arbitrarily close to zero.
On the other hand, if we assume that $\BA$ is constant and that either
$\GD = \emptyset$ or $\GN=\emptyset$, \eqref{eq_elliptic_regularity}
holds true with $\theta > 1/2$. Finally, if we further assume that
$\Omega$ is convex, \eqref{eq_elliptic_regularity} does hold true with $\theta = 1$.
We refer the reader to \cite{grisvard_1985a} for these last two statements.

As respectively constructed for instance in
\cite{ern_guermond_2017a,ern_gudi_smears_vohralik_2022a},
there exist (quasi-)interpolation operators
$\CI_h: L^2(\Omega) \to \CP_p(\CT_h) \cap H^1_{\GD}(\Omega)$
and
$\CJ_h: \BL^2(\Omega) \to \BCP_p(\CT_h) \cap \BX_{\GN}(\ddiv,\Omega)$
such that
\begin{subequations}
\label{eq_interpolation}
\begin{equation}
\label{eq_interpolation_H1}
\sum_{K \in \CT_h}
\left (
h_K^{-2} \|v-\CI_h v\|_K + \|\grad (v-\CI_h v)\|_K
\right )
\leq
Ch^{2\theta}|\grad v|_{H^\theta(\CT_h)}^2
\end{equation}
and
\begin{equation}
\label{eq_interpolation_Hdiv}
\sum_{K \in \CT_h}
\left (
\|\bw-\CJ_h \bw\|_K + h_K^2 \|\div (\bw-\CI_h \bw)\|_K
\right )
\leq
C \left (h^{2\theta}|\bw|_{H^\theta(\CT_h)}^2 + h^2\|\div \bw\|_\Omega\right )
\end{equation}
\end{subequations}
for all $v \in H^1_{\GD}(\Omega)$ and $\bw \in \BX_{\GN}(\ddiv,\Omega)$.

\begin{theorem}[Approximation factor]
\label{theorem_approximation_factor}
The following estimate holds true:
\begin{equation}
\label{eq_bound_gba}
\gba
\leq
C\gst
\left (\frac{\omega\ell}{\vel_{\min}}\right )^{1-\theta}
\left (\frac{\omega   h}{\vel_{\min}}\right )^{\theta}.
\end{equation}
\end{theorem}

\begin{proof}
Fix $\psi \in L^2(\Omega)$. On the one hand, we have
\begin{align*}
\|u_\psi^\star-\CI_h u_\psi^\star\|_{\dagger,1,\CT_h}^2
&=
\sum_{K \in \CT_h}
\left (
\frac{\alpha_K}{h_K^2} \|u_\psi^\star-\CI_h u_\psi^\star\|_K^2
+
\|\grad(u_\psi^\star-\CI_h u_\psi^\star)\|_{\BA,K}^2
\right )
\\
&\leq
\alpha_{\max}
\sum_{K \in \CT_h}
\left (
h_K^{-2}\|u_\psi^\star-\CI_h u_\psi^\star\|_K^2
+
\|\grad(u_\psi^\star-\CI_h u_\psi^\star)\|_{K}^2
\right )
\\
&\leq
C \alpha_{\max} h^{2\theta}|u_\psi^\star|_{H^{1+\theta}(\CT_h)}^2
\leq
C \frac{(\ell^{1-\theta}h^{\theta})^2}{\alpha_{\min}}
\|\div(\BA\grad u_\psi^\star)\|_{\Omega}^2
\end{align*}
due to \eqref{eq_interpolation_H1} and \eqref{eq_elliptic_regularity}.
Similarly, if we let $\bsig_\psi^\star \eq \BA\grad u_\psi^\star$, we have
\begin{align*}
\|\bsig^\star_\psi-\CJ_h\bsig^\star_\psi\|_{\dagger,\ddiv,\CT_h}^2
&=
\sum_{K \in \CT_h}
\left (
\|\bsig_\psi^\star-\CJ_h \bsig_\psi^\star\|_{\BA^{-1},K}^2
+
\frac{h_K^2}{\alpha_K}\|\div(\bsig_\psi^\star-\CJ_h \bsig_\psi^\star)\|_K^2
\right )
\\
&\leq
\frac{1}{\alpha_{\min}}
\sum_{K \in \CT_h}
\left (
\|\bsig_\psi^\star-\CJ_h \bsig_\psi^\star\|_K^2
+
h_K^2\|\div(\bsig_\psi^\star-\CJ_h \bsig_\psi^\star)\|_K^2
\right )
\\
&\leq
C
\frac{1}{\alpha_{\min}}
\left (
h^{2\theta} |\BA\grad u_\psi^\star|_{H^s(\CT_h)}^2
+
h^{2} \|\div(\BA\grad u_\psi^\star)|_{\Omega}^2
\right )
\end{align*}
where we employed \eqref{eq_interpolation_Hdiv} and \eqref{eq_elliptic_regularity}.
By definition of $u_\psi^\star$, we have
\begin{equation*}
\|\div(\BA\grad u_\psi^\star)\|_{\Omega}^2
=
\|\omega\mu \psi+\omega^2\mu u_\psi^\star\|_{\Omega}^2
\leq
C \mu_{\max} \omega^2
\left (
\|\psi\|_{\mu,\Omega}^2 + \omega^2\|u_\psi^\star\|_{\mu,\Omega}^2
\right )
\leq
C \gst^2 \mu_{\max} \omega^2 \|\psi\|_{\mu,\Omega}^2,
\end{equation*}
which leads to
\begin{equation*}
\|u_\psi^\star-\CI_h u_\psi^\star\|_{\dagger,1,\CT_h}
\leq
C \gst
\left (\frac{\omega\ell}{\vel_{\min}}\right )^{1-\theta}
\left (\frac{\omega h}{\vel_{\min}}\right )^{\theta}
\|\psi\|_{\mu,\Omega}
\end{equation*}
and
\begin{equation*}
\|\bsig_\psi^\star-\CJ_h\bsig_\psi^\star\|_{\dagger,\ddiv,\CT_h}
\leq
C\gst
\left (\frac{\omega\ell}{\vel_{\min}}\right )^{1-\theta}
\left (\frac{\omega h}{\vel_{\min}}\right )^{\theta}
\|\psi\|_{\mu,\Omega}.
\end{equation*}
Since they are valid for all $\psi \in L^2(\Omega)$,
these two estimates respectively enable to control
$\gbag$ and $\gbad$, leading to the bound on $\gba$
in \eqref{eq_bound_gba}.
\end{proof}

\begin{remark}[Frequency scaling]
In the setting considered here, it can be shown that
$\gst \sim \omega/\delta$, where $\delta$ is
the distance between $\omega$ and the closest resonant
frequency (the square-root of an eigenvalue). In this case,
under the assumption of full elliptic regularity, we obtain
\begin{equation*}
\gst \leq C \frac{\omega}{\delta} \frac{\omega h}{\vel_{\min}}.
\end{equation*}
This is consistent with previously obtained estimates,
see e.g. \cite{chaumontfrelet_vega_2022a}.
\end{remark}

\begin{remark}[Generalization to hanging nodes]
As can be seen by the proof of Theorem \ref{theorem_approximation_factor},
the approach proposed in this work does not need matching meshes, and
allows for hanging nodes as long as (quasi-)interpolation operators
satisfying \eqref{eq_interpolation} are available. For the full
elliptic regularity case where we can take $\theta=1$ in \eqref{eq_elliptic_regularity},
such interpolation operators are available for some families of meshes with hanging
nodes, see e.g. \cite[Lemma 12]{ainsworth_pinchedez_2002a}.
\end{remark}

\begin{remark}[Comparison with earlier estimates]
It is instructive to compare our results under minimal
regularity assumptions to the more standard approach
that requires more regularity. In this case, an approximation
factor similar to $\gbag$ is defined, but with an enhanced
norm including the normal trace of the gradient. This can
be seen, e.g., in respectively
\cite[Proposition 3.5]{melenk_parsania_sauter_2013a}
and
\cite[Definition 2.7]{sauter_zech_2015a}
for a priori and a posteriori error analysis.
If $p=1$ the estimate provided in
\cite[Theorem 4.8]{melenk_parsania_sauter_2013a}
for such an approximation factor has the same
scaling as the one derived here
(notice that the result in \cite{melenk_parsania_sauter_2013a}
is obtained under the assumption that $\gst \sim k^{1+\vartheta}$
for some $\vartheta \geq 0$ and that we can take $\theta=1$).
The results obtained in \cite{melenk_sauter_2010a}
and \cite{dorfler_sauter_2013a} for conforming discretizations
are also similar. We could also obtain sharp estimates for
high-order polynomials by using regularity splitting resulting
\cite{chaumontfrelet_nicaise_2020a,melenk_sauter_2010a}.
We refrain from doing so for the sake of brevity though.
\end{remark}
}

\end{document}

%% file: general_commands.tex
\usepackage{amssymb}
\usepackage{mathrsfs}

\renewcommand{\Re}{\operatorname{Re}}

\newcommand{\eq}{:=}

\newcommand{\grad}{\boldsymbol \nabla}
\renewcommand{\div}{\grad \cdot}

\newcommand{\ddiv}{\operatorname{div}}

\newcommand{\jmp}[1]{[\![#1]\!]}
\newcommand{\avg}[1]{\{\!\!\{#1\}\!\!\}}

\newcommand{\BA}{\boldsymbol A}

\newcommand{\BH}{\boldsymbol H}

\newcommand{\BL}{\boldsymbol L}

\newcommand{\BW}{\boldsymbol W}
\newcommand{\BX}{\boldsymbol X}

\newcommand{\bn}{\boldsymbol n}

\newcommand{\bv}{\boldsymbol v}
\newcommand{\bw}{\boldsymbol w}
\newcommand{\bx}{\boldsymbol x}
\newcommand{\by}{\boldsymbol y}

\newcommand{\CF}{\mathcal F}

\newcommand{\CI}{\mathcal I}
\newcommand{\CJ}{\mathcal J}

\newcommand{\CP}{\mathcal P}
\newcommand{\CQ}{\mathcal Q}

\newcommand{\CT}{\mathcal T}

\newcommand{\LP}{\mathscr P}
\newcommand{\LQ}{\mathscr Q}
\newcommand{\LR}{\mathscr R}

\newcommand{\TR}{\textup R}

\newcommand{\BCP}{\boldsymbol{\CP}}

\newcommand{\R}{\mathbb R}
\newcommand{\C}{\mathbb C}

%% file: special_commands.tex
\newcommand{\osc}{\operatorname{osc}}

\newcommand{\enorm}[1]{|\!|\!|#1|\!|\!|}

\newcommand{\bsig}{\boldsymbol \sigma}
\newcommand{\btau}{\boldsymbol \tau}

\newcommand{\LIFT}[1]{\boldsymbol{\ell}_h(\jmp{#1})}

\newcommand{\GRAD}{\mathfrak G}
\newcommand{\vel}{\vartheta}

\newcommand{\cgba}{\widecheck \gamma_{\rm ba}}

\newcommand{\cgbag}{\widecheck \gamma_{{\rm ba},{\rm g}}}
\newcommand{\cgbad}{\widecheck \gamma_{{\rm ba},{\rm d}}}

\newcommand{\tgba}{\widetilde \gamma_{\rm ba}}

\newcommand{\tgbag}{\widetilde \gamma_{{\rm ba},{\rm g}}}
\newcommand{\tgbad}{\widetilde \gamma_{{\rm ba},{\rm d}}}

\newcommand{\gba}{\gamma_{\rm ba}}
\newcommand{\gst}{\gamma_{\rm st}}

\newcommand{\gbag}{\gamma_{{\rm ba},{\rm g}}}
\newcommand{\gbad}{\gamma_{{\rm ba},{\rm d}}}

\newcommand{\pihg}{\pi_h^{\rm g}}
\newcommand{\pihd}{\pi_h^{\rm d}}

\newcommand{\pig}{\pi^{\rm g}}

\newcommand{\GD}{\Gamma_{\rm D}}
\newcommand{\GR}{\Gamma_{\rm R}}
\newcommand{\GN}{\Gamma_{\rm N}}

\newcommand{\bxi}{\boldsymbol \xi}